\documentclass[reqno,a4paper,12pt]{amsart} 

\usepackage{amsmath,amscd,amsfonts,amssymb}
\usepackage{mathrsfs,dsfont}

\usepackage{tikz}
\usetikzlibrary{decorations.pathreplacing}

\numberwithin{equation}{section}
\numberwithin{figure}{section}

\def\R{\mathbb{R}}

\def\Z{\mathbb{Z}}

\def\Gam{\Gamma}
\def\lam{\lambda}
\def\Lam{\Lambda}
\def\Om{\Omega}
\def\1{\mathds{1}}

\renewcommand\le{\leqslant}
\renewcommand\ge{\geqslant}
\renewcommand\leq{\leqslant}
\renewcommand\geq{\geqslant}

\renewcommand\hat{\widehat}

\newcommand{\ft}[1]{\widehat #1}
\newcommand{\dotprod}[2]{\langle #1 , #2 \rangle}
\newcommand{\mes}{\operatorname{mes}}

\newcommand{\dist}{\operatorname{dist}}

\newcommand{\zeros}{\mathcal{Z}}
\newcommand{\zft}[1]{\mathcal{Z}(\ft{\1}_{#1})}
\newcommand{\sqft}[1]{|{#1}|^{-2} \, |\ft{\1}_{#1}|^2}

\theoremstyle{plain}
\newtheorem{thm}{Theorem}[section]
\newtheorem{lem}[thm]{Lemma}
\newtheorem{corollary}[thm]{Corollary}

\newtheorem{problem}[thm]{Problem}
\newtheorem{conj}[thm]{Conjecture}

\newtheorem*{claim*}{Claim}

\newcommand{\thmref}[1]{Theorem~\ref{#1}}
\newcommand{\secref}[1]{Section~\ref{#1}}

\newcommand{\lemref}[1]{Lemma~\ref{#1}}
\newcommand{\defref}[1]{Definition~\ref{#1}}

\newcommand{\probref}[1]{Problem~\ref{#1}}
\newcommand{\corref}[1]{Corollary~\ref{#1}}
\newcommand{\conjref}[1]{Conjecture~\ref{#1}}

\theoremstyle{definition}
\newtheorem{definition}[thm]{Definition}
\newtheorem*{definition*}{Definition}
\newtheorem*{remarks*}{Remarks}
\newtheorem*{remark*}{Remark}
\newtheorem{remark}[thm]{Remark}

\newenvironment{enumerate-math}
{\begin{enumerate}
\addtolength{\itemsep}{5pt}
}
{\end{enumerate}}

\newenvironment{enumerate-text}
{\begin{enumerate}
\addtolength{\itemsep}{5pt}
}
{\end{enumerate}}

\begin{document}

 \title[Spectrality of product domains and Fuglede's conjecture]
{Spectrality of product domains and Fuglede's conjecture for convex polytopes}

\author{Rachel Greenfeld}
\address{Department of Mathematics, Bar-Ilan University, Ramat-Gan 5290002, Israel}
\email{rachelgrinf@gmail.com}
\author{Nir Lev}
\address{Department of Mathematics, Bar-Ilan University, Ramat-Gan 5290002, Israel}
\email{levnir@math.biu.ac.il}

\thanks{Research supported by ISF grant No.\ 227/17 and ERC Starting Grant No.\ 713927.}
\date{July 2, 2018}


\begin{abstract}
	A set $\Omega \subset \R^d$ is said to be spectral if the
	space $L^2(\Omega)$ has an orthogonal basis of exponential functions.
	 It is well-known that in many respects, spectral sets
	``behave like'' sets which can tile the space by translations.
	This suggests a conjecture that a product set $\Omega = A \times B$
	is spectral if and only if the factors $A$ and $B$ are both spectral sets.
	We recently proved  this in the case when $A$ is an interval in dimension one.
	The main result of the present paper is	
	that the conjecture is true also when $A$ is a convex polygon in two 
	dimensions. We discuss this result in connection with the conjecture that a convex 
	polytope $\Omega$ is  spectral  if and only if it can tile  by translations.
\end{abstract}

\maketitle


\section{Introduction} \label{secI1}

\subsection{}
Let $\Omega\subset \R^d$ be a bounded, measurable set of positive Lebesgue measure. 
It is said to be  \emph{spectral} if there exists a countable set $\Lambda\subset \R^d$
 such that the system of exponential functions
\begin{equation}
	\label{eqI1.1}
	E(\Lambda)=\{e_\lambda\}_{\lambda\in \Lambda}, \quad e_\lambda(x)=e^{2\pi
	i\dotprod{\lambda}{x}},
\end{equation}
is orthogonal and complete in  $L^2(\Omega)$,
  that is, the system is an orthogonal basis for the space.
Such a set  $\Lambda$ is called a \emph{spectrum} for $\Omega$.

The classical example of a spectral set is the unit cube $\Omega = \left[-\frac1{2}, \frac1{2}\right]^d$,
 for which  the set $\Lam = \Z^d$ serves as a spectrum.

Which other sets $\Omega$ are spectral? The study of this question was initiated by Fuglede in 1974
\cite{Fug74}, and it is known as \emph{Fuglede's spectral set problem}. 

The research on spectral sets has been motivated for many years by an observation due to  Fuglede, that the notion of spectrality is related to another, geometrical notion -- the tiling by translations. We say that $\Omega$ \emph{tiles} the space by translations along a countable set $\Lambda\subset \R^d$ if the collection of sets
$\{\Omega + \lam\}$, $\lam \in \Lam$, constitutes a partition of $\R^d$ up to measure zero.

With time, it became apparent that in many respects, spectral sets ``behave like'' sets which can tile the space by translations.
It was observed that many results about spectral sets have analogous results for sets which can tile, and vice versa.
However the precise connection between the notions of spectrality and tiling, is still not clear.

\subsection{}
 One of the interesting open problems in the subject is Fuglede's conjecture for convex bodies, which states that a convex body  $\Omega\subset \R^d$ is spectral if and only if it can tile the space by translations (originally this conjecture was stated in \cite{Fug74} for general, not necessarily convex sets  $\Omega$, but it turned out that in this generality the conjecture is not true; see \cite[Section 4]{KM10} and the references therein).

It has long been known that  a convex body   $\Omega$ which can tile by translations must be a polytope, and that it is a spectral set (see, for example, \cite[Section 3.5]{Kol04}). Much less is known, however, about the converse assertion. It was proved by Iosevich, Katz and Tao \cite{IKT03} that if a  convex polygon   $\Omega \subset \R^2$ is a spectral set,  then it must be either a parallelogram or a centrally symmetric hexagon, and hence it tiles by translations. Recently, we proved \cite{GL16, GL17} that Fuglede's conjecture is true also  for convex polytopes in dimension $d=3$. That is, if a convex polytope $\Omega \subset \R^3$   is  spectral, then it can tile by translations.

\subsection{}
 One of the difficulties in proving Fuglede's conjecture for convex polytopes in dimensions $d \geq 3$, is concerned with the existence of polytopes  $\Omega \subset \R^d$ which can be mapped by an invertible affine transformation to a cartesian product $A \times B$ of two convex polytopes $A \subset \R^n$,  $B \subset \R^m$  $(n,m \geq 1)$ where $n+m=d$. A convex polytope $\Omega$ with this property is said to be \emph{directly decomposable},  see \cite[Section~3.3.2]{Sch14}. For brevity, in this paper we will omit the word ``directly'', and just say that  $\Omega$  is \emph{decomposable}.

 For example, one can easily verify  that a two-dimensional convex polygon is decomposable if and only if it is a parallelogram, while in three dimensions a convex polytope is decomposable if and only if it is a prism.

The proof that a spectral convex polytope $\Omega$ in $\R^2$   or in $\R^3$    can necessarily  tile by translations, was based on the fact that if such an $\Omega$ is indecomposable, then it has a unique spectrum up to translation, see \cite[Theorems 1.3 and 1.4]{GL17}. The latter fact is no longer true if  $\Omega$  is  decomposable. Nevertheless, in two dimensions the situation when $\Omega$  is decomposable does not present any difficulty, as in this case $\Omega$  is a parallelogram and so it automatically tiles by translations. However, in dimensions $d \geq 3$,  decomposable polytopes do not necessarily tile. For this reason, the case when $\Omega$   is a prism in $\R^3$ required a different approach in our result, see \cite{GL16}.

\subsection{}
The study of decomposable spectral convex polytopes leads to the following, more general problem. Let $\Omega = A \times B$  be the cartesian product of two bounded, measurable sets  $A \subset \R^n$,  $B \subset \R^m$. When is $\Omega$  spectral? This question was posed in \cite{Kol16}. 

The answer is conjectured to be the following:

\begin{conj}
	\label{conjA1.1}
	Let $A \subset \R^n$ and $B \subset \R^m$   be two
	bounded, measurable sets. Then their product  $\Omega  = A \times B$
	is spectral if and only if $A$ and $B$ are both spectral sets.
\end{conj}

The ``if'' part of this conjecture is obvious. Indeed, if $U \subset \R^n$  is a spectrum for $A$,  
and $V \subset \R^m$  a spectrum for $B$,  then the
 product $\Lam = U \times V$  is a spectrum for $\Omega  = A \times B$ (see, for example, \cite{JP99}).
However the converse, ``only if'' part of the conjecture, is non-trivial. 
The difficulty lies in that we assume the product set $\Omega$ to be spectral, but
 we do not know that the spectrum $\Lambda$ also has a product structure. 
So it is not obvious which sets $U$ and $V$ may serve  as spectra for the factors $A$ and $B$, respectively.

One reason to expect that \conjref{conjA1.1} should be true is the fact that the analogous assertion for tiling by translations is known to hold. Indeed, it was observed in \cite[Section~1.2]{Kol16} that the product set  $\Omega  = A \times B$ can tile the space $\R^n \times \R^m$ by translations if and only if both $A$ tiles $\R^n$ and $B$ tiles $\R^m$. So the analogy between  spectrality and  tiling suggests  that \conjref{conjA1.1} should be true as well.

\subsection{}
In \cite{GL16} we proved the first result in the direction of \conjref{conjA1.1}: 

\begin{thm}[\cite{GL16}]
\label{thmA1.0}
Let $\Omega=A\times B$ where $A$ is an interval in $\R$,
	and $B$ is a bounded, measurable set in $\R^m$.
	Then $\Omega$ is spectral  if and only if $B$ is a spectral set.
\end{thm}

This result implies that \conjref{conjA1.1} is true whenever $A$ is a parallelepiped in $\R^n$.
Indeed, due to the invariance under affine transformations, it is enough to consider the case when
$A$ is the $n$-dimensional unit cube $\left[-\frac1{2}, \frac1{2}\right]^n$,
and the conclusion then follows from \thmref{thmA1.0} by induction on $n$.

\thmref{thmA1.0} played an important role in the proof of  Fuglede's conjecture for three-dimensional 
convex polytopes. It allowed us to use ``dimension reduction'' in order to resolve 
the case when $\Omega$ is decomposable -- that is, when $\Omega$ is a prism in $\R^3$. Indeed,
in this case we could assume, by applying an affine transformation, that
$\Omega$ is the cartesian product $A\times B$ of an interval $A \subset \R$,
 and  a convex polygon $B \subset \R^2$ (the polygon $B$ constitutes the \emph{base} of the prism).
By \thmref{thmA1.0}, the spectrality of $\Omega$ implies that $B$ must also be spectral,
 and  we could then invoke the two-dimensional result of \cite{IKT03} to conclude that $B$, and hence also $\Omega$,  tiles by translations
(see \cite[Section 9]{GL17}).

Kolountzakis found in \cite{Kol16} another proof of \thmref{thmA1.0},
 different from the one in \cite{GL16}. His approach moreover allowed him to establish that
\conjref{conjA1.1} is true also in  the case when the set $A$ is the union of two intervals in $\R$. 

\subsection{}
An important  special case of \conjref{conjA1.1}
is when the two sets $A, B$ are assumed to be convex polytopes.
A proof of the conjecture in this special case  
amounts to showing that the spectrality of a  decomposable  
convex polytope $\Omega$ can be characterized by 
 the spectrality of the factors  in the decomposition. Such a result
would reduce the proof of Fuglede's conjecture for convex polytopes
to the case when $\Omega$ is  indecomposable.

In this paper, our  main focus will be on the situation when $A$ is a convex polytope in $\R^n$,
while $B$ is an arbitrary bounded, measurable set in $\R^m$.


\section{Results}

\subsection{}
Our first result is concerned with necessary conditions for the spectrality of convex polytopes in $\R^n$.  By a result due to Kolountzakis \cite{Kol00a}, if a convex polytope $A \subset \R^n$ is spectral, then $A$ must be centrally symmetric. We proved in \cite{GL17} that also the central symmetry of all the facets of $A$ is a necessary condition for its spectrality.

The following theorem supports \conjref{conjA1.1} by showing that these conditions are necessary also for the spectrality of the product set $A\times B$.

\begin{thm}
	\label{thmC5.4}
	Let $\Omega=A\times B$ where $A$ is a convex polytope in $\R^n$, 
	and $B$ is a bounded, measurable set in $\R^m$.
	If $\Omega$ is a spectral set, then $A$ must be centrally symmetric and have centrally symmetric facets.
\end{thm}

The condition that the convex polytope $A$ is centrally symmetric and has centrally symmetric facets
is also necessary for $A$ to tile by translations, see \cite{McM80}.

\subsection{}
If $A$ is a convex body in $\R^n$ which is not a polytope, then $A$ cannot tile by translations,
see \cite{McM80}. It  is conjectured that such an $A$ can neither be spectral.
In this connection, a result from \cite{IKP99} states that if $A$ is a ball
in $\R^n$ $(n \geq 2)$ then $A$ is not a spectral set. In \cite{IKT01} the same was proved
for any centrally symmetric convex body $A$ with a smooth boundary.

The following theorem supports \conjref{conjA1.1} by extending these results to the context of product sets:

\begin{thm}
	\label{thmC5.5}
	Let $A$ be a centrally symmetric convex body in $\R^n$ 
	$(n \geq 2)$ with a smooth boundary,
	and $B$ be any bounded, measurable set in $\R^m$. Then 
	the product set $\Omega=A\times B$ cannot be  spectral. 
\end{thm}

\subsection{}
The  next theorem is the main result of this paper. The result confirms
that \conjref{conjA1.1} is true if $A$ is a convex polygon in two dimensions:

\begin{thm}
	\label{thmA1.1}
	Let $\Omega=A\times B$, where $A$ is a convex polygon in $\R^2$,
	and $B$ is a bounded, measurable set in $\R^m$.
	Then $\Omega$ is spectral  if and only if $A$ and $B$ are both spectral sets.
\end{thm}

If $A$  is a parallelogram, then this is a consequence of \thmref{thmA1.0}.
The new result is therefore that \thmref{thmA1.1} is true also if $A$ is
a convex polygon  which is not a parallelogram.
Our proof establishes that  in this case, $A$ must be a centrally symmetric hexagon. 
In particular, this implies the result from \cite{IKT03} that the 
 spectral convex polygons are exactly the parallelograms
and the centrally symmetric hexagons.

If both $A$ and $B$ are convex polygons in $\R^2$, then \thmref{thmA1.1}  implies
that their product $\Omega=A\times B$ is a spectral set if and only if $\Omega$ tiles by translations.
Combining this with the results obtained in \cite{GL16, GL17}, we can confirm that Fuglede's
conjecture is true for the class of decomposable convex polytopes in four dimensions:

\begin{corollary}
	\label{corA1.3}
	Let $\Omega \subset \R^4$  be a convex polytope, and assume 
	that $\Omega$ is decomposable. Then $\Omega$ is a spectral set if and only if it can tile by translations.
\end{corollary}

Thus, if we want to prove  Fuglede's conjecture for convex polytopes in dimension $d=4$,
then the case when $\Omega$ is decomposable is now covered by \corref{corA1.3},
and what remains to be proved is that   an indecomposable convex polytope
$\Omega \subset \R^4$  can be spectral only if it tiles by translations.
We will address this problem in a future work.


\section{Preliminaries} \label{secC1}

\subsection{Notation}
We use $\dotprod{\cdot}{\cdot}$ and $|\cdot|$ for the Euclidean scalar product and norm in $\R^d$.
\par
If $A \subset \R^d$ then $A^\complement$ denotes 
the complement  of $A$ (i.e.\ the set $\R^d \setminus A$),
$\1_A$ is the indicator function of $A$,
and $|A|$ or $\mes(A)$ is  the Lebesgue measure of $A$.
We use $A+B$, $A-B$ to denote the set of sums and set of 
differences of two sets $A, B \subset \R^d$.

If $f$ and $g$ are two measurable functions on $\R^n$ and $\R^m$ respectively,
then we denote by $f \otimes g$ the function on $\R^n \times \R^m$ defined by
$(f \otimes g)(x,y) = f(x)g(y)$.

\subsection{Spectra}
If $\Omega$ is a bounded, measurable set in $\R^d$ of positive measure, then
by a \emph{spectrum} for $\Omega$ we mean a countable set $\Lambda\subset\R^d$
such that the system of exponential functions $E(\Lambda)$ defined by \eqref{eqI1.1} is 
orthogonal and complete in the space $L^2(\Omega)$.

For any two points $\lam,\lam'$ in $\R^d$ we have
\[
		\dotprod{e_\lambda}{e_{\lambda'}}_{L^2(\Omega)} = \hat{\1}_\Omega(\lambda'-\lambda), 
\]
where
\[
\hat{\1}_\Omega(\xi) = \int_{\Omega} e^{-2\pi i\langle \xi,x\rangle} dx, \quad \xi \in \R^d,
\]
is the Fourier transform of the indicator function $\1_\Omega$ of the set $\Omega$.
The orthogonality of the system $E(\Lambda)$ in $L^2(\Omega)$ is therefore equivalent to the condition
\begin{equation}
	\label{eqP1.2}
	(\Lambda-\Lambda) \setminus \{0\} \subset  \zft{\Omega}, 
\end{equation}
where $\zft{\Omega} := \{ \xi \in \R^d : \hat{\1}_\Omega(\xi)=0\}$ is the set of zeros of the function $\hat{\1}_\Omega$.

The property of $\Lambda$ being a spectrum for $\Omega$ is
invariant under translations of both $\Omega$ and $\Lambda$.
If $M$ is a $d \times d$  invertible matrix, then
 $\Lambda$ is a spectrum for $\Omega$ if and only if
the set $(M^{-1})^\top (\Lambda)$ is a spectrum for $M(\Omega)$.

A set $\Lambda\subset \R^d$ is said to be \emph{uniformly discrete} if there is
$\delta>0$ such that $|\lambda'-\lambda|\ge \delta$ for any two distinct points
$\lambda,\lambda'$ in $\Lambda$. The maximal constant $\delta$ with this property is
called the \emph{separation constant} of $\Lambda$, and will be denoted by
$\delta(\Lambda)$.

The condition \eqref{eqP1.2} implies that every spectrum $\Lambda$ of $\Omega$
is a uniformly discrete set, and that its separation constant $\delta(\Lambda)$
is at least as large as the constant
\begin{equation}
	\label{eqP1.3}
	\chi(\Omega):= \min \big\{ |\xi| \;:\; \xi\in \zft{\Omega} \big\}> 0.
\end{equation}

\subsection{Weak limits}\label{secLimits}
Let $\Lambda_k$ be a sequence of uniformly discrete sets in $\R^d$, such that
$\delta(\Lambda_k)\ge\delta>0$. The sequence $\Lambda_k$ is said to \emph{converge
weakly} to a set $\Lambda$ if for every $\varepsilon>0$ and every $R$ there is
$N$ such that 
\[
\Lambda_k \cap B_R\subset \Lambda+B_\varepsilon \quad \text{and} \quad \Lambda\cap B_R\subset
\Lambda_k+B_\varepsilon
\]
 for all $k \ge N$, where $B_r$ denotes here the open ball of
radius $r$ centered at the origin. The weak limit $\Lambda$ is also a
uniformly discrete set, and satisfies $\delta(\Lambda)\ge \delta$. 

A compactness argument shows that  any sequence $\Lambda_k$ satisfying $\delta(\Lambda_k)\ge
\delta>0$, has a subsequence $\Lambda_{k_j}$ which converges weakly to some
(possibly empty) set $\Lambda$.

If for each $k$ the set $\Lambda_k$ is a spectrum for $\Omega$, and if $\Lambda_k$ converges
weakly to a limit $\Lambda$, then also $\Lambda$ is a spectrum for $\Omega$.
See \cite[Section 3]{GL16}.

\subsection{Tiling and packing}
Let $f \geq 0$ be a measurable function on $\R^d$, and $\Lam$ be a countable set in $\R^d$. 
We will say that $f+\Lam$ is a \emph{tiling} if the condition
\begin{equation}
\label{eqC3.1}
\sum_{\lam\in\Lam}f(x-\lam) = 1\quad\text{a.e.}
\end{equation}
is satisfied. If we only have
\begin{equation}
\label{eqC3.2}
\sum_{\lambda\in\Lambda}f(x-\lambda)\leq 1\quad\text{a.e.}
\end{equation}
then we will say that $f+\Lam$ is a \emph{packing}.

If $f=\1_\Omega$ is the indicator function of a bounded, measurable set
$\Omega\subset\R^d$, then the condition \eqref{eqC3.1} means that
the sets $\Omega+\lam$ $(\lam\in\Lam)$ constitute
a partition of $\R^d$ up to measure zero,
while \eqref{eqC3.2} says  that these sets are pairwise disjoint up to 
measure zero. In the former case we will
 say that  $\Omega+\Lam$  is a tiling, while in the latter we say
 that  $\Omega+\Lam$  is a packing.

The following lemma may be found,  for example, in \cite[Section 3.1]{Kol04}.
It gives a characterization of the spectra of $\Omega$, or the
exponential systems orthogonal in $L^2(\Omega)$, by
a tiling or a packing condition, respectively.

\begin{lem}
	\label{lemC1.7}
	Let $\Omega$ be a bounded, measurable set in $\R^d$, and define the function 
\[
f:= \sqft{\Omega}.
\]
\begin{enumerate-math}
	\item \label{lemC1.7.1}
	For a set $\Lam\subset\R^d$ to be a spectrum for $\Omega$ it is necessary and sufficient that $f+\Lam$ is a tiling.
	\item \label{lemC1.7.2}
	For a system of exponentials $E(\Lam)$ to be orthogonal in $L^2(\Omega)$ it is necessary and sufficient that $f+\Lam$ is a packing.
\end{enumerate-math}
\end{lem}

A proof of part \ref{lemC1.7.1} of \lemref{lemC1.7} is given also in \cite[Section 2.4]{GL16}.
The proof of part \ref{lemC1.7.2} is similar.

The next lemma may be found e.g.\ in \cite[Sections 1.1 and 3.3]{Kol04}.

\begin{lem}
	\label{lemC1.6}
	Let $f,g \in L^1(\R^d)$, $f,g \geq 0$. Assume that  $\Lam \subset \R^d$ is a set
	such that $f+\Lam$ is a tiling, while $g + \Lam$ is a packing. Then:
	\begin{enumerate-math}
		\item \label{lemC1.6.1}
		$\int g \leq \int f$.
		\item \label{lemC1.6.2}
		$\int g = \int f$ if and only if $g + \Lam$ is a tiling.
	\end{enumerate-math}
\end{lem}

\begin{proof}
	Since $f+\Lam$ is a tiling and $g + \Lam$ is a packing, we have
	\[
	f*\delta_\Lam=1 \quad \text{a.e.} , \qquad g*\delta_\Lam\leq 1 \quad \text{a.e.},
	\]
	where we denote
	\[
	\delta_\Lam := \sum_{\lam \in \Lam} \delta_\lam.
	\]
	The function $h := 1-g*\delta_\Lam$ is therefore nonnegative a.e., which implies that
	\[
	0\leq f \ast h =f \ast 1-f \ast (g\ast\delta_\Lam)=  f\ast 1 -  g\ast (f \ast \delta_\Lam) =\int f-\int g,
	\]
	and this proves \ref{lemC1.6.1}. 

	It also follows that 
	$\int g = \int f$ if and only if $f \ast h = 0$. But the
	convolution of two nonnegative functions cannot be everywhere zero,
	unless at least one of the functions vanishes a.e. Observe that $f$ cannot vanish a.e., 
	since $f+\Lam$ is a tiling. Hence $f \ast h = 0$ if and only if $h=0$ a.e., which means
	that  $g + \Lam$ is a tiling. This proves \ref{lemC1.6.2}. 
\end{proof}

\subsection{Definition}
If $W \subset \R^d$ is a bounded, measurable set, then we define
\begin{equation}
	\label{eqC2.1}
	\Delta(W) := \{x \in \R^d: \mes(W \cap (W + x)) > 0\}.
\end{equation}
The set $\Delta(W)$ is a bounded open set, symmetric with respect to the origin.

One can think of the set $\Delta(W)$ as the measure-theoretic analog of the
set of differences $W-W$. In particular, one can check that if $W$ is an open set then  $\Delta(W) = W-W$.
In general we  have $\Delta(W) \subset W-W$, but this inclusion can be strict. 

The following fact is easy to verify:
\begin{lem}
	\label{lemC2.3}
	Let $W$ be a bounded, measurable set in $\R^d$, 
	and $\Lam$ be a countable set in $\R^d$. Then
	$W + \Lambda$ is a  packing if and only if 
	$(\Lam-\Lam) \setminus \{0\} \subset \Delta(W)^\complement$.
\end{lem}

\subsection{Orthogonal packing regions}
\label{subsec:orthpackreg}

If $\Omega$ and $W$ are two  bounded, measurable sets in $\R^d$, then we will say
that $W$ is  an \emph{orthogonal packing region} for $\Omega$ if we have
\begin{equation}
	\label{eqC2.2}
	\Delta(W) \cap \zft{\Omega} = \emptyset.
\end{equation}

The notion of an orthogonal packing region was introduced in the paper
\cite{LRW00} but in less generality. In \cite{LRW00} it was additionally assumed 
that the boundary of $W$ is a set of measure zero, and instead of \eqref{eqC2.2}
the condition $(W^\circ-W^\circ)  \cap \zft{\Omega} = \emptyset$ was used
as the definition of  an orthogonal packing region, where $W^\circ$ denotes
the interior of $W$. In the present paper we use the definition \eqref{eqC2.2} to extend the notion
of an orthogonal packing region to the situation where $W$ is any  bounded, measurable set in $\R^d$.
It is easy to verify that the new definition coincides with the one given in  
\cite{LRW00} in the case when the boundary of $W$ is a set of measure zero.

The reason for the name ``orthogonal packing region'' is the fact that
if a system of exponentials $E(\Lam)$ is orthogonal in $L^2(\Om)$, and
if $W$ is  an orthogonal packing region for $\Omega$, then $W+\Lam$ is a packing.
This follows from \eqref{eqP1.2}, \eqref{eqC2.2} and \lemref{lemC2.3}.

\subsection{Convex polytopes}
By a \emph{convex polytope} $\Omega \subset \R^d$ we mean a compact
set which is the convex hull of a finite number of points.
By a \emph{facet} of $\Omega$  we refer to a $(d-1)$-dimensional face of $\Omega$.

We say that $\Omega$ is \emph{centrally symmetric} if $-\Omega$ is a translate
of $\Omega$. In this case, there is a unique point $x \in \R^d$ such that 
$-\Omega + x = \Omega - x$, and $\Omega$ is said to be symmetric with respect to the point $x$.

A convex polytope $\Omega \subset \R^d$  will be called  \emph{decomposable} if $\Omega$
can be mapped by an invertible affine transformation to a cartesian product $A \times B$
of two convex polytopes $A \subset \R^n$,  $B \subset \R^m$  $(n,m \geq 1)$ where $n+m=d$.
(Usually such a polytope is said to be ``directly decomposable'',  see e.g.\ \cite[Section 3.3.2]{Sch14},
but in this paper we use the term ``decomposable'' for brevity.)

If $\Omega$ is not decomposable, then we  say that $\Omega$ is \emph{indecomposable}.


\section{Kolountzakis' theorem} \label{secC4}

In this section we discuss a result of Kolountzakis, which gives a method for proving in certain
situations that  the spectrality of a product set $\Omega  = A \times B$ implies
the spectrality of the factors $A,B$. We give a simple proof of the result in a stronger form.

\subsection{}
	Let $A \subset \R^n$ and $B \subset \R^m$   be two
	bounded, measurable sets. 
\begin{definition}
	\label{defC4.1}
	Let $\Lam \subset \R^n \times \R^m$ be a spectrum for the product $\Omega  = A \times B$, and let
	$W \subset \R^n$ be a bounded, measurable set. We say that $\Lam$ is \emph{$W$-compatible} if the condition
\begin{equation}
	\label{eqC4.2}
	(\Lam-\Lam) \setminus \{0\} \subset (\Delta(W)^\complement \times \R^m) \cup (\R^n \times \zeros(\ft{\1}_B))
\end{equation}
is satisfied. 
\end{definition}

The condition \eqref{eqC4.2} can be  equivalently stated as follows:  given any pair
of distinct points $(u,v)$ and $(u',v')$  in $\Lam$,  the set
$(W+u) \cap (W+u')$ cannot have positive measure unless
the exponential functions $e_v$ and $e_{v'}$ are orthogonal in $L^2(B)$.

Notice that every spectrum $\Lam$ of $\Omega = A \times B$ satisfies the condition
\begin{equation}
	\label{eqC4.2.2}
	(\Lam-\Lam) \setminus \{0\} \subset \zft{\Omega} = (\zft{A} \times \R^m) \cup (\R^n \times \zft{B}),
\end{equation}
which follows from \eqref{eqP1.2}  and the fact that
$\ft{\1}_\Om = \ft{\1}_A \otimes \ft{\1}_B$.
This implies that condition \eqref{eqC4.2} holds whenever $W$ is an orthogonal packing region for $A$.

\subsection{}
The following result is basically due to Kolountzakis \cite{Kol16}.

\begin{thm}
	\label{thmC4.3}
	Let $A \subset \R^n$ and $B \subset \R^m$   be two
	bounded, measurable sets. Assume that  $\Omega  = A \times B$
	is a spectral set, and let $\Lam$ be a spectrum for $\Omega$.
	Suppose that there exists a bounded, measurable set
	$W \subset \R^n$, $|W| \geq |A|^{-1}$, such that $\Lam$ is  $W$-compatible.
	Then
	\begin{enumerate-math}
		\item \label{thmC4.3.1} $|W| = |A|^{-1}$;
		\item \label{thmC4.3.2} $B$ is a spectral set.
	\end{enumerate-math}
\end{thm}

This result was formulated in \cite[Theorem 2]{Kol16} in the special case when
$W$ is assumed to be an orthogonal packing region for $A$. This assumption 
implies that every spectrum $\Lam$ of $\Omega   = A \times B$
is  $W$-compatible. So in this case the result says:\footnote{Strictly speaking, in \cite{Kol16} the notion of an orthogonal 
packing region for $\Omega$ was defined using the condition $(W-W) \cap \zft{\Omega} = \emptyset$,
so formally a special case of \corref{corD4.3.1} was proved in \cite{Kol16}.
However the proof in \cite{Kol16} can be easily extended to the situation in the present
paper, where an orthogonal packing  region for $\Omega$ is defined using  condition \eqref{eqC2.2}.
Moreover, if $W$ is an open set, then the two definitions of an orthogonal packing region 
used in \cite{Kol16} and in the present paper coincide.}

\begin{corollary}[\cite{Kol16}]
	\label{corD4.3.1}
	Let $\Omega  = A \times B$ be the product of two
	bounded, measurable sets $A \subset \R^n$ and $B \subset \R^m$.
	Suppose that $A$ has an orthogonal packing region $W$,
	$|W| \geq |A|^{-1}$. If $\Omega$
	is spectral, then  conclusions  \ref{thmC4.3.1} and \ref{thmC4.3.2} in \thmref{thmC4.3} are true.
\end{corollary}

However the proof  in  \cite{Kol16} of this result in fact uses the assumption that $W$ is an orthogonal 
packing region for $A$ only to ensure that condition \eqref{eqC4.2} holds. So actually  \thmref{thmC4.3}
follows from that proof.

If $A \subset \R$ is an interval, then any interval $W$ of length 
$|W| = |A|^{-1}$ is an orthogonal packing region for $A$. Hence
\corref{corD4.3.1} can be used in this case to conclude that the spectrality of $\Omega  = A \times B$ implies
the spectrality of $B$.  This yields \thmref{thmA1.0}.

In general, however, it is not obvious how to use \thmref{thmC4.3} in order to prove
that the spectrality of a given product set $\Omega  = A \times B$ implies
the spectrality of the factors $A,B$. Indeed, to apply this theorem
one must first establish the existence of 
a set $W \subset \R^n$, $|W| \geq |A|^{-1}$,
and of a spectrum $\Lam$ for $\Omega$,
such that $\Lam$ is  $W$-compatible. This would imply the spectrality of $B$ by
part \ref{thmC4.3.2} of the theorem. Secondly, to conclude that also $A$ must
be spectral, one must show in addition that if this is not the case then $W$ can
be chosen such that $|W| > |A|^{-1}$. Then part \ref{thmC4.3.1}
of  \thmref{thmC4.3} would lead to a contradiction.

Kolountzakis showed \cite[pp.\ 107--108]{Kol16} that if $A \subset \R$
 is the union of two intervals then $A$ admits 
an orthogonal packing region $W$, such that $|W| = |A|^{-1}$ if $A$ is spectral,
while $|W| > |A|^{-1}$ otherwise. Thus
\corref{corD4.3.1} can be applied in this situation, to conclude
that the spectrality of $\Omega  = A \times B$ implies
the spectrality of both $A$ and $B$, see \cite[Corollary 5]{Kol16}.

\subsection{}
Kolountzakis' proof involves a construction which is often  referred to as ``cut-and-project''.

Let $p_1$ and $p_2$ denote  the projections from $\R^n \times \R^m$
onto $\R^n$ and $\R^m$ respectively, that is, $p_1(u,v)=u$ and $p_2(u,v)=v$.

\begin{definition}
\label{defD4.1}
Assume that
 $\Lam \subset \R^n \times \R^m$ is a countable set, and that $W \subset \R^n$ is a bounded, measurable set. Then the set
\begin{equation}
	\label{eqD4.1.1}
 \Gamma(\Lam, W) := p_2(\Lam \cap (W \times \R^m)) \subset \R^m
\end{equation}
will be called \emph{the cut-and-project set based on $\Lam$ and $W$}
(see Figure \ref{fig:cutandproject}).
\end{definition}


\begin{figure}[htb]
\centering
\begin{tikzpicture}[scale=0.375]

\fill [fill=gray!35] (7,0) -- (11,0) -- (11,18) --  (7,18) -- (7,0);

\draw [loosely dashed] (7,0) -- (7,18);
\draw [loosely dashed] (11,0) -- (11,18);

\draw[-stealth, densely dotted, black!75!white] (8,4) -- (0.25,4);
\draw[-stealth, densely dotted, black!75!white] (10,6) -- (0.25,6);
\draw[-stealth, densely dotted, black!75!white] (9,10) -- (0.25,10);
\draw[-stealth, densely dotted, black!75!white] (8,11) -- (0.25,11);
\draw[-stealth, densely dotted, black!75!white] (9,16) -- (0.25,16);

\fill (0,4) circle (0.15);
\fill (0,6) circle (0.15);
\fill (0,10) circle (0.15);
\fill (0,11) circle (0.15);
\fill (0,16) circle (0.15);

\fill (8,4) circle (0.1);
\fill (10,6) circle (0.1);
\fill (9,10) circle (0.1);
\fill (8,11) circle (0.1);
\fill (9,16) circle (0.1);

\fill (14,3) circle (0.1);
\fill (5,7) circle (0.1);
\fill (12,8) circle (0.1);
\fill (15,9) circle (0.1);
\fill (4,12) circle (0.1);
\fill (5,14) circle (0.1);
\fill (13,15) circle (0.1);

\draw[-stealth]  (0, 0) -- (18, 0);
\draw[-stealth]  (0, 0) -- (0, 18);
\draw (18,0) node[anchor=west] {$\R^n$};
\draw (0,18) node[anchor=south] {$\R^m$};

\draw[line width=0.5mm] (7,0) -- (11,0);
\draw (9,-1) node {\small $W$};

\draw (-1,10) node[anchor=east] {\small $\Gamma(\Lam, W)$};

\end{tikzpicture}
\caption{The cut-and-project set  $\Gamma(\Lam, W)$.}
\label{fig:cutandproject}
\end{figure}
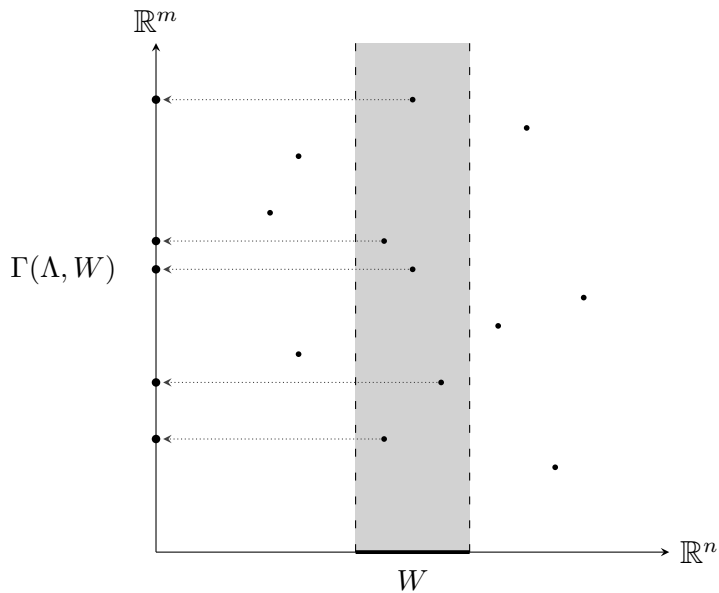


(It should be remarked that in the literature, by a ``cut-and-project'' construction one usually refers to the special
situation where the set $\Lam$ is assumed to be a lattice. However we do not make such an assumption in \defref{defD4.1}).

In many situations when working with cut-and-project sets, it is natural to impose the extra
assumption that the projection $p_2$ is a one-to-one map when
restricted to the set $\Lam \cap (W \times \R^m)$. 
This  means that for every $v \in \Gamma(\Lam, W)$ there exists a unique
$u \in W$ such that the point $(u,v)$ belongs to $\Lam$.

Now suppose that  $\Omega  = A \times B$ is the product of
two	bounded, measurable sets $A \subset \R^n$ and $B \subset \R^m$.
Kolountzakis observed that the assumptions in \thmref{thmC4.3} imply that
 for a.e.\ $x \in \R^n$  the projection $p_2$ is one-to-one on the set 
$\Lam \cap ((W+x) \times \R^m)$, and its image $\Gamma(\Lam, W + x)$ constitutes a set of frequencies
 in $\R^m$ whose corresponding exponential system $E(\Gamma(\Lam, W + x))$
is orthogonal in $L^2(B)$. Moreover, he showed that there exist choices of $x$
such that in addition, the so-called \emph{upper uniform density} of the set
$\Gamma(\Lam, W + x) $ is bounded from below by values arbitrarily close to $|A| \cdot |B| \cdot |W|$
(for the definition of the upper uniform density, see \cite[p.\ 100]{Kol16}).
Finally, Kolountzakis proved a result of independent interest \cite[Theorem 1]{Kol16}
which implies that the latter fact suffices to establish the conclusion
of \thmref{thmC4.3}.

\subsection{}
In what follows, we give a simple proof of \thmref{thmC4.3}.
The proof moreover establishes a new conclusion
about the structure of the spectrum $\Lam$:

\begin{thm}
	\label{thmC4.4}
	Under the same assumptions as in \thmref{thmC4.3}, the following conclusion is true:
		for a.e.\ $x \in \R^n$ the projection $p_2$ is one-to-one on the set 
		$\Lam \cap ((W+x) \times \R^m)$ and its image 
		$\Gamma(\Lam, W + x) $ is a spectrum for the set $B$.
\end{thm}

In other words, the new result is that for a.e.\ $x \in \R^n$, not only the exponential
system $E(\Gamma(\Lam, W + x))$ is orthogonal in $L^2(B)$, but this system
is also complete in the space.

\begin{proof}[Proof of both \thmref{thmC4.3} and \thmref{thmC4.4}]
We divide the proof into several steps. 

\emph{Step 1}. 
We  show that for a.e.\ $x \in \R^n$, the map $p_2$ is one-to-one
on the set $\Lam \cap ((W+x) \times \R^m)$, and the exponential system
$E(\Gamma(\Lam, W + x))$ is orthogonal in $L^2(B)$.

 This amounts to showing that 
the set
\begin{equation}
\label{eqC4.4.1.0}
\begin{aligned}
\{x \in \R^n : \; &\text{there exist distinct points $(u,v)$, $(u',v')$ in $\Lam$}\\
&\text{such that $u, u' \in W+x$ but $v' - v \notin \zft{B}$}\}
\end{aligned}
\end{equation}
has measure zero.
To show this, let $(u,v)$, $(u',v')$ be  two distinct points in $\Lam$ such
that $v' - v \notin \zft{B}$. Since $\Lam$ was assumed to be  $W$-compatible,
it follows from condition \eqref{eqC4.2} that $u' - u \notin \Delta(W)$. Using the definition
\eqref{eqC2.1} of the set $\Delta(W)$, this implies that
\[
\mes \{x \in \R^n: u, u' \in W+x\} = \mes ((W-u) \cap (W-u')) = 0.
\]
Since  $\Lam$ is a countable set, this shows that the set in \eqref{eqC4.4.1.0}
can be decomposed into a countable union of sets of measure zero, and hence this set itself also has measure zero, as we had to prove.

In what follows, we denote $f_A := \sqft{A}$ and $f_B := \sqft{B}$.

\emph{Step 2}. 
We show that $(\1_{-W}\otimes f_B)+\Lam$ is a packing. 

This means that the sum
	\begin{equation}
	\label{eqC4.4.1.1}
	\sum_{(u,v)\in\Lam}\1_{-W}(x-u)f_B(y-v)
	\end{equation}
	should be not greater than $1$ for a.e.\ $(x,y)\in\R^n\times\R^m$.

	Assume that $x \in \R^n$ is a point lying outside the set in \eqref{eqC4.4.1.0}. 
	Since the map $p_2$ is one-to-one on the set $\Lam \cap ((W+x) \times \R^m)$, the sum in 
	\eqref{eqC4.4.1.1} is equal to
	\begin{equation}
	\label{eqC4.4.1.2}
	\sum_{v\in\Gamma(\Lam, W + x) }f_B(y-v).
	\end{equation}
	Since the system  
	$E(\Gamma(\Lam, W + x))$ is orthogonal in $L^2(B)$, it follows from
	part \ref{lemC1.7.2} of \lemref{lemC1.7} that the sum in 
	\eqref{eqC4.4.1.2} is not greater than $1$
 for a.e.\ $y \in \R^m$. The claim thus follows from Fubini's theorem.

\emph{Step 3}. 
We show that $(f_A \otimes f_B)+\Lam$ is a tiling. 

This is a  consequence of part \ref{lemC1.7.1} of \lemref{lemC1.7}, since
$f_A \otimes f_B = \sqft{\Om}$ and the set $\Lam$ is a spectrum for $\Om$.

\emph{Step 4}. 
We show that $|W|=|A|^{-1}$.

The fact that $(f_A\otimes f_B) +\Lam$ is a tiling and $(\1_{-W}\otimes f_B)+\Lam$ is a packing,
allows us to use \lemref{lemC1.6}. It follows from part \ref{lemC1.6.1} of the lemma that
	\begin{equation}	
	\label{eqC4.4.10.1}
	 |W| \cdot |B|^{-1}=\iint_{\R^n \times \R^m} \1_{-W}\otimes f_B \leq \iint_{\R^n \times \R^m} f_A\otimes f_B= |A|^{-1} \cdot |B|^{-1},
	\end{equation}
	and so we obtain $|W|\leq |A|^{-1}$. However, we  assumed a priori that $|W|\geq |A|^{-1}$, so the equality  $|W|=|A|^{-1}$ must hold. This establishes  part \ref{thmC4.3.1} of \thmref{thmC4.3}.

\emph{Step 5}. 
We show that $(\1_{-W}\otimes f_B)+\Lam$ is a tiling. 

Indeed, since $|W|=|A|^{-1}$ we see that the inequality in
\eqref{eqC4.4.10.1} is in fact an equality. Hence the claim follows from
 part \ref{lemC1.6.2} of \lemref{lemC1.6}.

\emph{Step 6}. 
We  show that for a.e.\ $x \in \R^n$, the 
set $\Gam(\Lam, W+x)$ is a spectrum for $B$.

We have seen that, for a.e.\ $(x,y)\in\R^n\times\R^m$, the sums in \eqref{eqC4.4.1.1} 
and \eqref{eqC4.4.1.2} coincide, and
the sum in \eqref{eqC4.4.1.1} is equal to $1$ since $(\1_{-W}\otimes f_B)+\Lam$ is a tiling.
So a further application of Fubini's theorem yields that for 
a.e.\ $x \in \R^n$ there is a set $Y(x) \subset \R^m$ of full measure, such that
the sum in \eqref{eqC4.4.1.2} is equal to $1$ for all $y \in Y(x)$.
Hence for a.e.\ $x \in \R^n$ we have that $f_B + \Gamma(\Lam, W + x)$ is a tiling, 
and we conclude from part \ref{lemC1.7.1} of \lemref{lemC1.7} 
that the set $\Gam(\Lam, W+x)$ is a spectrum for $B$.
This establishes \thmref{thmC4.4} and in particular also part \ref{thmC4.3.2} of \thmref{thmC4.3}.
\end{proof}


\section{Orthogonal exponentials and relatively dense sets} \label{secC5}

In this section our main goal is to prove \thmref{thmC5.4}, which says that if a product set
$\Omega=A\times B$ is spectral, and if the set $A$ is a convex polytope 
in $\R^n$, then $A$ must be centrally symmetric and have centrally symmetric facets.
We also prove \thmref{thmC5.5}.

\subsection{}
A set $\Lam \subset \R^d$  is said to be \emph{relatively dense} if
there is $R = R(\Lam) >0$ such that every ball of radius $R$ contains at least one point from $\Lam$.

It is known that if $\Lam$ is a spectrum for some bounded, measurable set
$\Omega  \subset \R^d$ then $\Lam$
must be a relatively dense set (see, for example, \cite[Section 2C]{GL17}).

\begin{lem}
	\label{lemC5.3}
	Let $A \subset \R^n$ and $B \subset \R^m$ be two bounded, measurable sets, and suppose
	that their product $\Om=A\times B$ is a spectral set.
	Then there exists a relatively dense set $\Gamma \subset \R^n$
	such that the system of exponentials $E(\Gamma)$ is orthogonal in $L^2(A)$.
\end{lem}

One can  view this lemma as establishing  a weak form of \conjref{conjA1.1}.
Indeed, the conjecture asserts that the spectrality of $\Om=A\times B$ implies
the existence of a spectrum $\Gamma$ for $A$. The relative denseness of 
$\Gam$ and the orthogonality of the system $E(\Gamma)$ in $L^2(A)$
are necessary conditions for $\Gam$ to be a spectrum for $A$. However, these conditions are not
sufficient, since they do not guarantee that the system $E(\Gamma)$  is also complete in $L^2(A)$.

\begin{proof}[Proof of \lemref{lemC5.3}]
Assume that the assertion of the lemma is not true, so that 
every set $\Gamma \subset \R^n$ for which the system  $E(\Gamma)$ is orthogonal in $L^2(A)$,
 is not relatively dense. We will show that this leads to a contradiction. 

First, we will establish the following:

\begin{claim*}
Let $G$ be an open set in $\R^m$, and $V$ be an open ball in $\R^m$ of diameter $\chi(B)$
$($where $\chi(B)$ is defined as in \eqref{eqP1.3}$)$.
Suppose that  $\Om$ has a spectrum $\Lam$ satisfying
\begin{equation}
\label{eqC5.3.6}
\Lam\cap(\R^n\times G)=\emptyset.
\end{equation}
Then there exists also a spectrum $\Lam'$ for $\Om$  such that
\begin{equation}
\label{eqC5.3.9.1}
\Lam'\cap(\R^n\times (G\cup V))=\emptyset.
\end{equation}
\end{claim*}	

Indeed, consider
the set\footnote{One may notice that the set $\Gam$ in \eqref{eqC5.3.9} is a cut-and-project set as in Definition \ref{defD4.1}, except that in the present case the roles of $p_1$ and $p_2$ are interchanged.} 
\begin{equation}
\label{eqC5.3.9}
\Gam:=\{ u \in \R^n : \text{there is $v \in V$ such that $(u,v)\in \Lam$}\}.
\end{equation}
Let us show that the system $E(\Gam)$ is  orthogonal in  $L^2(A)$. To see this, let 
$u, u'$ be two distinct elements in $\Gam$. Then by \eqref{eqC5.3.9} there exist $v, v' \in V$ 
such that the (distinct) points $\lambda:=(u,v)$ and $\lambda':=(u',v')$ are in $\Lam$. 
Hence $\lam' - \lam \in \zft{\Om}$. Since
$\ft{\1}_\Om = \ft{\1}_A \otimes \ft{\1}_B$,
we must have $u'-u \in \zft{A}$ or $v'-v \in \zft{B}$. But as $v, v'$ both lie in the open ball $V$ of
diameter $\chi(B)$, the latter possibility cannot occur. It follows that
$u'-u \in \zft{A}$, which means that the exponentials $e_u$ and $e_{u'}$ are 
orthogonal in  $L^2(A)$. 

Once we know that the system $E(\Gam)$ is  orthogonal in  $L^2(A)$, it follows that the set $\Gam$  cannot be relatively dense. Hence
	there is a sequence $t_j$ of vectors in $\R^n$ satisfying
	\begin{equation}
	\label{eqC5.3.4}
	\Gam\cap(Q_j+ t_j)=\emptyset,
	\end{equation}
	where $Q_j := (-j, j)^n$ denotes the open cube in $\R^n$ of side length $2j$ centered at the origin.
	By the definition \eqref{eqC5.3.9} of $\Gam$, the condition \eqref{eqC5.3.4} means that
		\begin{equation}
		\label{eqC5.3.7}
		\Lam\cap((Q_j+t_j)\times V)=\emptyset.
		\end{equation}
	Define
	\[
		\Lam_j:=\Lam-(t_j,0).
	\]
	Then $\Lam_j$ is also a spectrum for $\Omega$. It follows from \eqref{eqC5.3.6} and \eqref{eqC5.3.7} that
	\begin{equation}
	\label{eqC5.3.8}
	\Lam_j\cap(Q_j\times(G\cup V))=\emptyset,
	\end{equation}
	for every $j$. We may extract  from the sequence $\Lam_j$ a weakly convergent subsequence, 
	whose limit $\Lam'$ is also a spectrum for $\Om$. The condition \eqref{eqC5.3.8} guarantees that 
	the new spectrum $\Lam'$ satisfies \eqref{eqC5.3.9.1}. The claim is therefore proved.
	
	Next, we will conclude the proof of \lemref{lemC5.3} based on the above claim. 
	We choose a sequence of open balls $V_k \subset \R^m$ whose union covers the whole $\R^m$,
	and such that each $V_k$ has  diameter $\chi(B)$.
	We construct inductively a sequence $\Lam_k$ of spectra for $\Omega$, such that
	\begin{equation}
	\label{eqC5.3.3}
	\Lam_k \cap(\R^n\times (V_1 \cup \cdots \cup V_k))=\emptyset
	\end{equation}
	for each $k$. The construction is done
	as follows. 
	We start by taking $\Lam_0$ to be any spectrum for $\Omega$. Then, in the $k$'th step
	of the construction, we apply the claim with $\Lam=\Lam_{k-1}$, $G=V_1 \cup \cdots \cup V_{k-1}$ and $V=V_k$.
	The claim yields a spectrum $\Lam_k$ for $\Omega$ that satisfies \eqref{eqC5.3.3}, as required.
	
	Finally  observe that, since the balls  $V_k$ cover the whole $\R^m$, it follows from
	\eqref{eqC5.3.3} that the sequence $\Lam_k$  converges weakly to the empty set.
	This is a contradiction, since a weak limit of spectra for $\Om$ must 
	be a spectrum as well. \lemref{lemC5.3} is thus proved.
\end{proof}

\begin{remark}
\label{remC5.3.1}
The conclusion of \lemref{lemC5.3} remains true if we relax
the assumption that the product set $\Omega=A\times B$ is spectral,
and instead only require the existence of a relatively 
dense set $\Lam \subset \R^n \times \R^m$ such that the system 
$E(\Lam)$ is orthogonal in $L^2(\Omega)$. Indeed, it is not difficult to
check that the proof given above remains valid under this weaker assumption.
\end{remark}

\subsection{}
Now assume that $A$ is a convex polytope in $\R^n$. In order to prove
\thmref{thmC5.4} we will rely on the following two results.

\begin{thm}[Kolountzakis \cite{Kol00a}]
\label{thmC6.1}
Let $A$ be a convex polytope in $\R^n$.
If $A$ is a spectral set, then $A$ is centrally symmetric.\footnote
{Actually it was proved in \cite{Kol00a} that any convex body   (not assumed to be a polytope) in $\R^n$
which is spectral, must be centrally symmetric.}
\end{thm}

\begin{thm}[\cite{GL17}]
\label{thmC6.2}
Let $A$ be a convex, centrally symmetric polytope in $\R^n$.
If $A$ is spectral, then all the facets of  $A$ are also  centrally symmetric.
\end{thm}

Suppose now that the product set
$\Omega=A\times B$ is spectral. If we knew that the spectrality
of $\Omega$ implies that $A$ must also be spectral, then we 
could deduce from Theorems
\ref{thmC6.1} and \ref{thmC6.2}  that $A$ is centrally symmetric
and has centrally symmetric facets. Recall that
by \lemref{lemC5.3}, there is a relatively dense set
$\Gamma \subset \R^n$ such that the system of exponentials $E(\Gamma)$ 
is orthogonal in $L^2(A)$. Nevertheless, this does not mean that $A$ is spectral,
since we do not know that the system $E(\Gamma)$  is
also complete in $L^2(A)$.
We therefore cannot formally deduce \thmref{thmC5.4} based on the statements
of \lemref{lemC5.3} and Theorems \ref{thmC6.1} and  \ref{thmC6.2}.

However there is a proof of \thmref{thmC6.1} 
which in fact does not require the
completeness of the system $E(\Gamma)$  in $L^2(A)$,  only
its orthogonality and the relative denseness of $\Gam$
\cite{KP02} (see also \cite[Section 3]{GL17}).
The same is true also for the proof of \thmref{thmC6.2}, see \cite[Section 4]{GL17}. 
Hence a consequence of these proofs is that the following more general version of the results is actually true:

\begin{thm}
\label{thmC5.1}
Let $A$ be a convex polytope in $\R^n$. Assume that there exists a relatively dense set $\Gamma \subset \R^n$ such that the system of exponentials $E(\Gamma)$ is orthogonal in $L^2(A)$. Then $A$ must be centrally symmetric and have centrally symmetric facets.
\end{thm}

This more general version is suitable for combining with \lemref{lemC5.3}, and this yields
\thmref{thmC5.4} as an immediate corollary.

\begin{remark}
The conclusion of Theorems \ref{thmC6.1} and \ref{thmC6.2} (or \thmref{thmC5.1})
cannot be further improved by showing
that also all the $k$-dimensional faces of  $A$, for some $2 \le k \le n-2$, 
must be  centrally symmetric (see \cite[Section~4A]{GL17}).
\end{remark}

\begin{remark}
In the special case when the set $B$ is also a convex polytope, we can derive \thmref{thmC5.4} directly from
Theorems \ref{thmC6.1} and  \ref{thmC6.2}, without the use of \lemref{lemC5.3}.
Indeed, in this case the product $\Omega = A \times B$ is a convex polytope in $\R^n\times \R^m$.
Hence $\Omega$ must be centrally symmetric (\thmref{thmC6.1}), which means that
the set $-\Omega=(-A)\times (-B)$ is a translate of $\Omega$. It follows that $-A$ is a translate of $A$,
thus $A$ is  centrally symmetric. 
Next, suppose that  $F$ is a facet of $A$. Then $F\times B$ is a facet of $\Omega$, hence it
is also centrally symmetric (\thmref{thmC6.2}). Thus $(-F)\times (-B)$ is a
translate of $F\times B$, and as before we can deduce that $F$ is centrally
symmetric. 
\end{remark}

\begin{remark}
For $n \geq 2$, the convex polytopes $A \subset \R^n$ which satisfy
the assumptions in \thmref{thmC5.1} form  a strictly larger class  than the spectral convex polytopes. 
As an example, let $P$ be any convex, 
centrally symmetric polygon in $\R^2$, which is neither a parallelogram nor a  hexagon,
and whose vertices lie in $\Z^2$. Then the zero set $\zft{P}$ contains
$\Z^2 \setminus \{0\}$ (this follows, for instance, from the results in \cite{Kol00b}). Hence if we
define $A := P \times [0,1]^{n-2}$, then $A$ is a convex polytope in $\R^n$
such that the system $E(\Z^n)$ is orthogonal in $L^2(A)$.
However $A$ is not spectral, by \thmref{thmA1.0} and the fact
 (due to \cite{IKT03}) that $P$ is not spectral.
\end{remark}

\subsection{}
\label{secC5.2}
There is another result about spectral sets, whose proof in fact  requires only the
existence of an orthogonal (but not necessarily complete) system 
of exponentials $E(\Gamma)$ with a relatively dense
set of frequencies $\Gam$. The result states that a ball
\cite{IKP99}, or more generally, a centrally symmetric convex body with a smooth boundary
\cite{IKT01} in $\R^n$ $(n \geq 2)$ cannot be spectral. The proof of this result shows that 
the following more general version is true:

\begin{thm}
\label{thmC5.2}
Let $A$ be a centrally symmetric convex body in $\R^n$ $(n \geq 2)$ with a smooth boundary. 
Then there cannot exist a relatively dense set $\Gamma \subset \R^n$ 
such that the system of exponentials $E(\Gamma)$ is orthogonal in $L^2(A)$. 
\end{thm}

Combining \lemref{lemC5.3} and  \thmref{thmC5.2} we conclude that
the product $\Omega=A\times B$ of  a centrally symmetric
convex body $A \subset \R^n$ $(n \geq 2)$ with a smooth boundary,
and a bounded, measurable set $B \subset \R^m$, can never be  spectral.
This yields \thmref{thmC5.5}.


\section{Convex polygons: Effective constraints for spectra} 
\label{secD1}

One of the difficulties in the spectral set problem for convex polytopes $\Omega$
involves the structure of the zero set  $\zft{\Omega}$.
If $\Lambda$ is a spectrum for $\Omega$, then this zero set 
appears in  condition \eqref{eqP1.2}
which is equivalent to the orthogonality of the exponential system $E(\Lam)$ in 
the space $L^2(\Omega)$. However, unless $\Omega$ is a parallelepiped, 
there is no explicit description of the zero set  $\zft{\Omega}$, a fact which 
presents a certain  difficultly in using the condition  \eqref{eqP1.2}  effectively. 

Our purpose in this section is to circumvent this difficulty  in the case when
$\Omega=A\times B$ is the product of a
convex polygon $A \subset \R^2$, and a bounded, measurable set  $B \subset \R^m$. 
We do this by introducing a new set that we denote by  $H(A)$, and which can be used 
in some sense as a substitute for the zero set  $\zft{A}$.

Using the asymptotics of the Fourier transform
$\ft{\1}_A$ we will prove that if
$\Omega = A \times B$ is a spectral set, then it has a spectrum $\Lam$ 
satisfying a version of  condition \eqref{eqP1.2}, obtained by replacing 
the  zero set  $\zft{A}$ with  the new set $H(A)$.
The advantage of the set  $H(A)$ is that it can be explicitly calculated,
which provides us with more effective information on the structure
of the spectrum $\Lam$. This additional information will be used in 
\secref{secD2} to prove  the main result of the paper, \thmref{thmA1.1}.

\subsection{}
	Let $A \subset \R^2$ be a convex, centrally symmetric polygon.
If $e$ is any edge of  $A$, then by the central symmetry there is another edge $e'$ of  $A$ which is parallel to $e$ and has the same length. Hence $e$ is a translate of $e'$, so there is a translation vector $\tau_e$ in $\R^2$ which carries $e'$ onto $e$
(see Figure \ref{fig:taue}). Define
\begin{equation}
	\label{defD1.3}
	H(A, e) := \{t \in \R^2 : \dotprod{t}{\tau_e} \in \Z
	\; \; \text{or} \; \;  \dotprod{t}{e} \in \Z \setminus \{0\} \},
\end{equation}
where we use $e$ also to denote a vector in $\R^2$
which has the same direction and length as the edge $e$
(such a vector is unique up to a sign).
The set $H(A, e)$ consists of an infinite
system of straight lines with directions perpendicular either to $\tau_e$ or to $e$.


\begin{figure}[htb]
\centering
\begin{tikzpicture}[scale=0.8, p2/.style={line width=0.275mm, black}, p3/.style={line width=0.15mm, black!50!white}]

\draw[p2] (-1, -2) -- (2, -2);
\draw[p3] (2,-2) -- (3.5,-1) -- (3.5,0) -- (2.8, 1) -- (1, 2);
\draw[p2] (1, 2) -- (-2,2);
\draw[p3] (-2,2) -- (-3.5, 1) -- (-3.5, 0) -- (-2.8, -1) -- (-1,-2);

\draw[-latex, p2] (-1,-2) -- (-2,2);

\draw (-0.5, 2) node[anchor=south]{$e$};
\draw (0.5, -2) node[anchor=north]{$e'$};
\draw (-1.5, 0) node[anchor=west]{$\tau_e$};
\end{tikzpicture}
\caption{Two parallel edges $e$ and $e'$ of $A$, and the vector $\tau_e$.}
\label{fig:taue}
\end{figure}
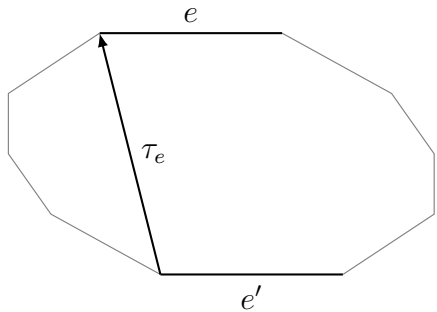


The sets $H(A, e)$ are used to construct another set $H(A)$, defined as follows:

\begin{definition}
\label{defD1.4}
We denote
\begin{equation}
	\label{eqdefD1.4}
	H(A) := \bigcap_{e}  H(A, e) \setminus \{0\},
\end{equation}
where the intersection is taken over all the edges $e$ of $A$.
\end{definition}

For example, suppose that $A$ is a parallelogram
spanned by two linearly independent vectors $a$, $b$ in $\R^2$.
Then it is not difficult to check that 
\begin{equation}
\label{eqD1.3.1}
H(A) = \{t \in \R^2 : \dotprod{t}{a} \in \Z \setminus \{0\}
\; \; \text{or} \; \; \dotprod{t}{b} \in \Z \setminus \{0\} \}.
\end{equation}
The set $H(A)$ thus consists of infinitely many
straight lines with directions  perpendicular to one of the edges of $A$.
Notice  that in this case we have $H(A) = \zft{A}$.

On the other hand, if $A$ is not a parallelogram, then one can verify that
$H(A)$ is a discrete closed set,
contained in the union of a finite number of lattices. This observation can essentially 
be found in \cite{Kol00b}. (The paper \cite{Kol00b} was concerned with a different 
subject -- the  structure of multi-tilings of $\R^2$ by translates of polygonal regions,
but interestingly the set which we denote by $H(A)$ was used in that paper as well.)

It is obvious that the set $H(A)$ remains invariant under translations of $A$.
It is also easy to check that if $M$ is a $2 \times 2$ invertible matrix
 then $H(M(A)) = (M^{-1})^\top (H(A))$.

\subsection{}
Before we move on to study the spectrality of the product set $\Omega=A\times B$,
 we will first illustrate the idea of how to use the set $H(A)$  in connection with a simpler 
problem, namely, the spectrality of the convex polygon $A$ itself.

It was proved in \cite{IKT03} that the spectral convex polygons in $\R^2$
are precisely the parallelograms and the centrally symmetric hexagons. 
Another proof of this fact was given in \cite[Section~8]{GL17}. In this paper
 we will obtain a third proof of this result. The proof relies on the following lemma:

\begin{lem}
	\label{lemD1.5}
	Let $A$ be a convex, centrally symmetric polygon in $\R^2$. If
	$A$ is a spectral set, then it admits a spectrum $\Gamma$ satisfying
\begin{equation}
	\label{eqD1.5}
	(\Gamma - \Gamma) \setminus \{0\} \subset H(A).
\end{equation}
\end{lem}

In order to understand the point of this lemma, 
recall that every spectrum $\Gamma$ of $A$ satisfies the condition 
$(\Gamma-\Gamma) \setminus \{0\} \subset \zft{A}$.
\lemref{lemD1.5} asserts  that at least one spectrum $\Gam$  exists for which 
the alternative condition \eqref{eqD1.5}, obtained by replacing
the zero set $\zft{A}$ with the set $H(A)$,  is also satisfied.
We will see in  \secref{secD2} that this alternative condition
can be used in order to deduce that $A$ must be  either a parallelogram or a centrally symmetric hexagon.

\begin{proof}[Proof of \lemref{lemD1.5}]
	First we show that it will be enough to prove the following:
	\begin{claim*}
		Assume that $e$ is an edge of $A$, and that $\Gam$ is a spectrum for $A$. Then there exists a sequence of translates of $\Gam$, which converges weakly to a spectrum  $\Gam'$ of $A$ satisfying the condition 
		\begin{equation}
		\label{eqD1.5.7}\Gam'-\Gam'\subset H(A,e).
		\end{equation}
	\end{claim*}
	Indeed, if this claim is true, then  \lemref{lemD1.5} can be established as follows.
	We enumerate all the edges of $A$ as 
	$e_1,e_2,\dots,e_N$.
	We let $\Gam_0$ be any spectrum of $A$, and apply the claim with $e=e_1$ and $\Gam=\Gam_0$. The claim yields a spectrum $\Gam_1$ of $A$, satisfying $\Gam_1-\Gam_1\subset H(A,e_1)$. We then apply again the claim 
	with $e=e_{2}$ and $\Gam=\Gam_{1}$, and obtain a spectrum $\Gam_{2}$ for $A$  such that
	$\Gam_{2}-\Gam_{2}\subset H(A,e_{2})$. 
	Moreover, as $\Gam_2$ is a weak limit of translates of  $\Gam_{1}$, the set $\Gam_2-\Gam_2$ is contained in the closure of $\Gam_1-\Gam_1$. Since $\Gam_1-\Gam_1\subset  H(A,e_1)$ and the set $H(A,e_1)$ is closed, we deduce that $\Gam_{2}-\Gam_{2}\subset  H(A,e_1)$ as well. 
	We continue applying the claim with $e=e_3$ and $\Gam=\Gam_2$, and so on.
	At the $k$'th step we  obtain a spectrum $\Gam_k$ for $A$, which satisfies $\Gam_k-\Gam_k\subset H(A,e_j)$ for all $1\leq j\leq k$.  Then the set $\Gam:=\Gam_N $, obtained after $N$ steps, would satisfy \eqref{eqD1.5} as needed.

	It therefore remains to prove the claim.
	To begin, notice that by applying an affine transformation we may assume
	that $A$ is symmetric about the origin, that the points $(\frac1{2},-\frac1{2})$
	and $(\frac1{2},\frac1{2})$ are vertices of $A$, and that the edge $e$ is 
	the line segment which connects these two points.
	These assumptions imply, by \eqref{defD1.3}, that
	\begin{equation}
	\label{eqD1.5.3}
	H(A,e)=(\Z\times\R)\cup(\R\times(\Z\setminus\{0\})).
	\end{equation}
	
	Consider the sequence  $\Gam_k:=\Gam-(k,0)$, $k=1,2,3,\dots$, of translates of $\Gam$. From this sequence we may extract  a weakly convergent subsequence, whose limit $\Gam'$ is also a spectrum for $A$.
	There is therefore an infinite set $S$ of positive integers, such that $\Gam_k\to \Gam'$  as $k\to\infty$,  $k\in S$. We will show that $\Gam'$ satisfies \eqref{eqD1.5.7}.

	To show this, we let  $(u',v')$ and $(u'',v'')$ be two points in $\Gam'$ and we need to verify that $(u''-u',v''-v')$ belongs to the set $H(A,e)$ given in \eqref{eqD1.5.3}. We will assume that  $v''-v'\not\in\Z\setminus\{0\}$ and prove that this implies that  $ u'' -u' \in \Z$. 
	We choose two sequences $k_j'$ and $k''_j$ in $S$, such that each one of them tends  to infinity and  also $k''_j-k'_j\to \infty$. Since $\Gam'$ is the weak limit of the sequence $\Gam_k$, $k\in S$,  there exist two sequences $(u'_j,v'_j)$ and $(u''_j,v''_j)$ in $\Gam$  such that 
	\begin{equation}
	\label{eqD1.5.5}
	(u'_j-k'_j,v'_j)\to (u',v')\quad\text{and}\quad (u''_j-k''_j,v''_j)\to (u'',v'').
	\end{equation} 
	Then, as we have chosen  $k'_j$ and $k''_j$ to satisfy $k''_j-k'_j\to \infty$, we obtain that also $u''_j-u'_j\to \infty$.
	In particular, for all large enough $j$, the points $(u'_j,v'_j)$ and $(u''_j,v''_j)$ are distinct, and since they both belong to $\Gam$, we have
	\begin{equation}
	\label{eqD1.5.2}
	\ft{\1}_A(u''_j-u'_j,v''_j-v'_j )=0.
	\end{equation} 
	
	We will now use the fact that for any fixed $C>0$,
	\begin{equation}
	\pi u  \hat{\1}_A(u,v)= \sin \pi u\cdot
	\ft{\1}_I(v) + O( |u|^{-1} ), \quad
	|u|\to \infty,\quad |v|\leq C,
	\label{eqD1.5.1}
	\end{equation}
	where $I$ denotes the interval $[-\frac{1}{2},\frac{1}{2}]$
	(a proof of this fact can be found in \cite[Lemma 6.1]{GL17}).
	Notice that \eqref{eqD1.5.5} implies that there is $C>0$ such that $|v''_j-v'_j|\leq C$ for all $j$.
	Thus, as we have also seen  that  $u''_j-u'_j\to \infty$, we may use  \eqref{eqD1.5.1},
	which implies by \eqref{eqD1.5.2} that
	\[\sin \pi (u''_j-u'_j)\cdot \ft{\1}_I (v''_j-v'_j) \to 0.\] 
	
	Recall now that we have assumed that $v''-v'\not\in\Z\setminus\{0\}$.  Then, as $\zft{I}=\Z\setminus\{0\}$, we obtain from \eqref{eqD1.5.5} that $|\ft{\1}_I(v''_j-v'_j)|$  remains bounded away from zero as $j\to\infty$.	
	It follows  that 
	$ \sin \pi(u''_j-u'_j)\to 0$, or equivalently, 
	$ \dist(u''_j-u'_j,\Z)\to 0$.
	Notice that as  both $k'_j$ and $k''_j$ are  integers, we have
	\begin{equation}
	\label{eqD1.5.4}
	\dist\left( u''-u',\Z\right)\leq \dist\left(u''_j-u'_j,\Z\right)+|u'_j-k'_j-u'|+|u''_j-k''_j-u''|.
	\end{equation}
	It follows that $u''-u'\in \Z$, since by
	combining  both \eqref{eqD1.5.5} and the fact that $ \dist(u''_j-u'_j,\Z)\to 0$ we obtain that the right hand side of \eqref{eqD1.5.4} tends to zero. This establishes  the claim, and thus completes the proof of \lemref{lemD1.5}.
\end{proof}

\begin{remark}
\label{remD1.5}
The claim in the proof of \lemref{lemD1.5} is actually a special case of a more general result, which can be stated in any dimension, and which was proved in  \cite[Sections 6, 7]{GL17}. 
\end{remark}

\subsection{}
Now we turn to the spectrality problem for the product set 
$\Omega = A \times B$. In this case we know that
every spectrum $\Lam$ of $\Omega$ satisfies the condition 
\begin{equation}
	\label{eqD2.1.2}
	(\Lam-\Lam) \setminus \{0\} \subset \zft{\Omega} = (\zft{A} \times \R^m) \cup (\R^2 \times \zft{B}).
\end{equation}

If we replace in \eqref{eqD2.1.2}  the zero set $\zft{A}$ with the set $H(A)$, then we obtain a new condition:
\begin{equation}
	\label{eqD2.1}
		(\Lam-\Lam) \setminus \{0\}  \subset (H(A) \times \R^m) \cup (\R^2 \times \zeros(\ft{\1}_B)).
\end{equation}
We will show that this condition is satisfied by at least one spectrum $\Lambda$ of $\Omega$:

\begin{lem}
	\label{lemD2.1}
	Let $A$ be a convex, centrally symmetric polygon in $\R^2$, and $B$ be a
	bounded, measurable set in $\R^m$. If 
	$\Omega = A \times B$ is a  spectral set, then there is a spectrum $\Lam$ of $\Omega$ which satisfies condition \eqref{eqD2.1}.
\end{lem}

This lemma will be used in \secref{secD2} in order to prove \thmref{thmA1.1}.

\begin{proof}[Proof of \lemref{lemD2.1}]
The proof is very similar to that of \lemref{lemD1.5}, and so it will only be sketched. 
First we show that the lemma can be reduced to the following claim:
\begin{claim*}
	If $e$ is an edge of $A$, and if $\Lambda$ is a spectrum of $\Omega = A \times B$,
	then there is a sequence of translates of $\Lambda$ which converges weakly to a 
	spectrum $\Lambda'$ of $\Omega$ satisfying
	\begin{equation}
	\label{eqD2.1.1}
	\Lambda' - \Lambda' \subset (H(A,e) \times \R^m) \cup (\R^2 \times \zeros(\ft{\1}_B)).
	\end{equation}
\end{claim*}
If this claim is true then, by iterating through all the edges $e_1, e_2, \ldots ,e_N$ of A,  
we can obtain a spectrum $\Lambda$ for $\Omega$
such that $\Lam - \Lam$ is contained simultaneously in all the sets
$(H(A,e_j) \times \R^m) \cup (\R^2 \times \zeros(\ft{\1}_B))$, $1 \leq j \leq N$.
Moreover,  $\Lam$ satisfies also condition \eqref{eqD2.1.2} which is true for any 
spectrum of $\Omega$. Notice that we have
\[
H(A,e_1) \cap H(A,e_2) \cap \cdots \cap H(A,e_N) \cap \zft{A} \subset H(A),
\]
which is a consequence of \eqref{eqdefD1.4} and the fact that the
zero set $\zft{A}$ does not contain the origin. Combining together all the mentioned properties
yields that $\Lam$ satisfies condition \eqref{eqD2.1} as required.

 We now turn to prove the claim. We may assume that $A$ is symmetric about the origin, and that $e=\{\frac{1}{2}\}\times [-\frac{1}{2},\frac{1}{2}]$. Then we have $H(A,e)=(\Z\times\R)\cup(\R\times(\Z\setminus\{0\})).$
   Let $\Lam'$ be a spectrum of $\Om$, obtained as  a weak limit of some subsequence of \[\Lam_k:=\Lam-(k,0,0,\dots,0), \quad k=1,2,3,\dots.\] That is, $\Lam_k\to\Lam'$ as $k\to\infty$, $k\in S$, where $S$ is a certain infinite set of positive integers.
   We will show that the spectrum $\Lam'$ satisfies \eqref{eqD2.1.1}.

To show this, we assume that  
   \[(u',v',w'),(u'',v'',w'')\in\R\times\R\times\R^m\] are two points in $\Lam'$, such that $v''-v'\not\in\Z\setminus\{0\}$ and also $w''-w'\not\in\zft{B}$. We then need to verify that   $u''-u'\in\Z$. 
   Let   $k'_j, k''_j$ be two sequences in $S$ satisfying $k'_j\to\infty,\; k''_j\to\infty,\; k''_j-k'_j\to \infty$.
   Then there exist two corresponding sequences $(u'_j,v'_j,w'_j)$, 
$(u''_j,v''_j,w''_j)$ in $\Lam$ such that
   \begin{equation}
   \label{eqD2.1.3}
   (u'_j-k'_j,v'_j,w'_j)\to (u',v',w')\quad\text{and}\quad (u''_j-k''_j,v''_j,w''_j)\to (u'',v'',w'').
   \end{equation} 
  From \eqref{eqD2.1.2} we deduce that
    for all large enough $j$, either $(u''_j-u'_j,v''_j-v'_j )\in\zft{A}$  or $w''_j-w'_j\in\zft{B}$. However, combining \eqref{eqD2.1.3} with the  assumption that $w''-w'\not\in\zft{B}$, we obtain  that $w''_j-w'_j\not\in\zft{B}$ for all large enough $j$, and so, we must have  
   $\ft{\1}_A(u''_j-u'_j,v''_j-v'_j )=0$.  
One can now continue in the same way as in the proof of \lemref{lemD1.5}, and establish that $u''-u'\in\Z$, which completes the proof.
\end{proof}


\section{Convex polygons: The notion of a window} \label{secD2}

In this section we obtain the main result of the paper, \thmref{thmA1.1}.
The theorem states that if  $\Omega=A\times B$ is the product of  a convex polygon $A \subset \R^2$,
and  a bounded, measurable set $B \subset \R^m$, then
$\Omega$ is  spectral  if and only if $A$ and $B$ are both spectral sets.
In other words, \conjref{conjA1.1} is true if $A$ is a convex polygon in two dimensions.

The non-trivial part is to prove that the spectrality of $\Omega=A\times B$
implies that both $A$ and $B$ must be spectral. We therefore assume
that $\Omega$ is a spectral set. 
By \thmref{thmC5.4} we know that in this case,  the convex polygon $A$ must be centrally symmetric.

Now suppose we knew that $A$ has an orthogonal packing region $W$,
 $|W| \geq |A|^{-1}$, and that moreover, if $A$ is neither
a parallelogram nor a hexagon,  then $W$ can be chosen such that $|W| > |A|^{-1}$. 
In such a  case, we would be able to use \corref{corD4.3.1} to 
conclude that both $A$ and $B$ must be spectral sets.

However the problem with this strategy is that, unless $A$ is a parallelogram,  no such an 
orthogonal packing region $W$ is known to exist for $A$. In fact we find the existence of
such a $W$ unlikely, although no proof of this claim is known to us either.

To resolve this problem we will introduce a new notion
that we call a ``window'', and which  replaces the notion of an  orthogonal  packing region in our context.
The advantage of this new notion is that we can prove
that if  $A$ is a convex,  centrally symmetric polygon, then it has a window $W$ such that
 $|W| \geq |A|^{-1}$, and if $A$ is neither
a parallelogram nor a hexagon,   then  $|W| > |A|^{-1}$. 

Moreover, we will see that if $\Omega=A\times B$ is a spectral set,
then it has at least one spectrum $\Lam$ which is $W$-compatible
(in the sense of \defref{defC4.1}). This allows us to apply  \thmref{thmC4.3}
in order to conclude that $A$ must be either a parallelogram or a hexagon, and
hence $A$ is a spectral set; and moreover, that the set $B$ must also be spectral.
Thus we will obtain  \thmref{thmA1.1}.

\subsection{}
First of all we need to define, what do we mean by a ``window''.

\begin{definition}
	\label{defD1.1}
	Let $A \subset \R^2$ be a convex, centrally symmetric polygon.
	We say that a bounded, measurable set $W \subset \R^2$ is a \emph{window} for
	$A$ if the condition
\begin{equation}
	\label{defD1.2}
	\Delta (W) \cap H(A) = \emptyset
\end{equation}
is satisfied.
\end{definition}

The set $H(A)$ was defined in \secref{secD1} (see \defref{defD1.4}).
If we replace the set $H(A)$ in condition \eqref{defD1.2} with the zero set
$\zft{A}$, then the condition becomes the definition of an orthogonal packing
region for $A$. Thus our notion of a window differs from that 
of an orthogonal packing region in that the set $H(A)$ replaces the zero set
$\zft{A}$.

Recall that if $A$ is a parallelogram then $H(A) = \zft{A}$.
Hence, for a parallelogram the notion of a window coincides with that of
an orthogonal packing region.

\subsection{}
The following result establishes a key fact concerning the notion of a window.
\begin{thm}
\label{thmD1.7}
Let $A$ be a convex, centrally symmetric polygon in $\R^2$. Then
	\begin{enumerate-math}
		\item \label{thmD1.7.1}
		If $A$ is  a parallelogram or a hexagon, then it has a window $W$,  $|W| = |A|^{-1}$;
		\item \label{thmD1.7.2}
		Otherwise, $A$ admits a window $W$ such that $|W| > |A|^{-1}$.
	\end{enumerate-math}
\end{thm}

As we have mentioned, it is not known whether this result remains true
if we replace the word ``window''  with ``orthogonal packing region'' in the statement.
Hence the fact that we can prove \thmref{thmD1.7} constitutes  the
main advantage of the  notion of a  window 
over that of an orthogonal  packing region in the context of convex polygons.

\begin{proof}[Proof of \thmref{thmD1.7}] 
	First observe that if  the assertions \ref{thmD1.7.1} and \ref{thmD1.7.2}  of the theorem are true for a certain convex, centrally symmetric polygon $A$, then these assertions are true also for any polygon $A'$ which is the image of $A$ under an invertible affine transformation. Indeed, we may write $A'=M(A)+x$, where $M$ is a $2 \times 2$ invertible matrix, and $x$ is a vector in $\R^2$.
	Assuming that assertion \ref{thmD1.7.1} or \ref{thmD1.7.2} is true for $A$ with some window $W$, it follows that the corresponding assertion for $A'$ is true with the window $W':=(M^{-1})^\top(W)$.
	
	It will  therefore suffice that we prove \ref{thmD1.7.1} and \ref{thmD1.7.2} under the following additional assumptions: the polygon $A$ is symmetric about the origin, two of its vertices lie at the points $(\frac1{2},-\frac1{2})$ and  $(\frac1{2},\frac1{2})$, and the line segment that connects these two points is an edge of $A$. We will denote this edge by $e_1$.
	We will also denote by $e_2$  the edge of $A$ that shares the vertex $\left(\frac1{2}, \frac1{2}\right)$ with $e_1$, and by $\left( a,b\right) $ the other vertex of $A$ that lies on $e_2$.  See Figure \ref{fig:polygon}.


\begin{figure}[htb]
\centering
\begin{tikzpicture}[scale=0.425]
\fill [fill=gray!35] (5,-5) -- (5,5) -- (3,8) --  (-5,5) -- (-5,-5) -- (-3,-8) -- (5,-5);
\draw  (5,-5) -- (5,5) -- (3,8) -- (-1,9) -- (-4,7) --  (-5,5) -- (-5,-5) -- (-3,-8) -- (1,-9) -- (4, -7) -- (5,-5);

\draw [loosely dashed] (3,8) -- (-5,5);
\draw [loosely dashed] (-3,-8) -- (5,-5);

\fill (5,5) circle (0.1);
\fill (-5,5) circle (0.1);
\fill (5,-5) circle (0.1);
\fill (-5,-5) circle (0.1);
\draw (5,5) node[anchor=west]  {\small $(\tfrac{1}{2},\tfrac{1}{2})$};
\draw (-5,5) node[anchor=east]  {\small $(-\tfrac{1}{2},\tfrac{1}{2})$};
\draw (5,-5) node[anchor=west] {\small $(\tfrac{1}{2},-\tfrac{1}{2})$};
\draw (-5,-5) node[anchor=east]  {\small  $(-\tfrac{1}{2},-\tfrac{1}{2})$};

\fill (3,8) circle (0.1);
\draw (3,8) node[anchor=south west] {\small $(a,b)$};
\fill (-3,-8) circle (0.1);
\draw (-3,-8) node[anchor=north east] {\small $(-a,-b)$};

\draw [decorate,decoration={brace,raise=3pt,amplitude=4pt,aspect=0.6}, gray!90]
	(5,-4.75) -- (5,4.75) node [gray!130,midway,xshift=-0.5cm,yshift=0.4cm] {\footnotesize  $e_1$};
\draw [decorate,decoration={brace,raise=3pt,amplitude=4pt}, gray!90]
	(4.9,5.15) -- (3.1,7.85) node [gray!130,midway,xshift=-0.4cm,yshift=-0.3cm] {\footnotesize $e_2$};

\draw[-stealth]  (-8, 0) -- (8, 0);
\draw[-stealth]  (0, -11) -- (0, 11);

\end{tikzpicture}
\caption{The convex polygon $A$ in the proof of \thmref{thmD1.7}.}
\label{fig:polygon}
\end{figure}


	Notice that if $A$ is a parallelogram, then it follows from the assumptions above that $A$ is in fact the unit cube, $A=[-\tfrac{1}{2},\tfrac{1}{2}]^2$. In this case we have $(a,b)=\left(-\frac1{2}, \frac1{2}\right)$.

	Now, to prove \thmref{thmD1.7}, it is required to show that there is a bounded, measurable set $W$ in $\R^2$ satisfying:
	\begin{enumerate-text}
		\item 
		$W$ is a window for $A$.
		\item
		If $A$ is a parallelogram or a hexagon then $|W|=|A|^{-1}$, and otherwise $|W|>|A|^{-1}$.
	\end{enumerate-text}
	We claim that the rectangle 
	\begin{equation}
	\label{eqD1.7.7}
	W:=\left\lbrace  (u,v)\in \R^2 :\; |u|<\tfrac1{2},\; |v|<\tfrac1{2}\left( b+\tfrac1{2}\right) ^{-1}  \right\rbrace 
	\end{equation}
	satisfies these conditions.
	
	We will first prove that the rectangle $W$ in \eqref{eqD1.7.7} is a window for $A$. To show this, we need to verify  that \eqref{defD1.2} holds. Notice that as $W$   is an open set, we have $\Delta (W)=W-W$, which in turn implies that $\Delta (W)=2W$, since $W$ is convex and  symmetric about the origin. This gives
	\begin{equation}
	\label{eqD1.7.3}
	\Delta (W)=\left\lbrace  (u,v)\in \R^2 :\; |u|<1,\; |v|< \left(b+ \tfrac1{2}\right) ^{-1} \right\rbrace .
	\end{equation}
	It is then required to verify that   $\Delta(W)$  is disjoint from the set $H(A)$.
	In fact, we will prove that
	\begin{equation}
	\label{eqD1.7.1}
	\Delta (W)\cap H(A,e_1)\cap H(A,e_2)\setminus\{0\}=\emptyset .
	\end{equation} 
	According to Definition \ref{defD1.4}, the set $H(A)$ is contained in  $H(A,e_1)\cap H(A,e_2)\setminus\{0\}$. Therefore \eqref{defD1.2} would follow from \eqref{eqD1.7.1}.

	To establish \eqref{eqD1.7.1}, we assume that  $(u,v)$ is a vector lying in $\Delta (W)\cap H(A,e_1)\cap H(A,e_2)$, and we must show that $(u,v)$ is the zero vector.  
	First we will use the fact that $(u,v)\in H(A,e_1)$.
	Notice that we have $\tau_{e_1}= (1,0)$ and   
	$e_1=(0,1)$ (where here $e_1$ is regarded as a vector in $\R^2$). 
	Hence, by the definition \eqref{defD1.3} of $H(A,e_1)$ we have
	\begin{equation}
	\label{eqD1.7.10}
	u\in\Z\;\;\text{or}\;\;v\in\Z\setminus\{0\}.
	\end{equation}
	Moreover, as the vector $(u,v)$  belongs also to the set $\Delta(W)$, it follows from \eqref{eqD1.7.3}  that
	\begin{equation}
	\label{eqD1.7.11}
	|u|<1\;\;\text{and}\;\;|v|<(b+\tfrac{1}{2})^{-1}\leq 1,
	\end{equation}
	where the last inequality is due to the fact that $b\geq\frac1{2}$, by the convexity of $A$.
	By combining \eqref{eqD1.7.10} and \eqref{eqD1.7.11}, we obtain that $u=0$.
	
	Next we use the fact  that  the vector $(u,v)$  lies in $H(A,e_2)$ as well.
	Notice that $ \tau_{e_2}=\left( a+\frac1{2},b+\frac1{2}\right)$ and 
	$e_2=\left( a-\frac1{2},b-\frac1{2}\right)$ (where, as before, $e_2$ is regarded here as a vector). Hence, since we have seen that $u=0$, we have
	\[  v(b+\tfrac1{2})=\dotprod{(u,v)}{\left( a+\tfrac1{2},b+\tfrac1{2}\right)} \in \Z, \]
	or
	\[ v(b-\tfrac1{2})=\dotprod{(u,v)}{\left( a-\tfrac1{2},b-\tfrac1{2}\right)} \in \Z \setminus \{0\}.\]
	Therefore, if $v\not =0$, then in each one of these cases we obtain that $|v|\geq (b+\tfrac1{2})^{-1}$, which contradicts  \eqref{eqD1.7.11}. We conclude  that   the set $\Delta(W) \cap H(A,e_1)\cap H(A,e_2)$ does not contain any vector other than the zero vector. This establishes \eqref{eqD1.7.1}, which, as we have seen, implies that  $W$ is a window for $A$.
	
	For our claim to be proved, it remains to show that the rectangle $W$ in \eqref{eqD1.7.7}  satisfies $|W|= |A|^{-1}$ if $A$ is a parallelogram or a hexagon, and $|W|>|A|^{-1}$ otherwise.
	To begin, notice that the definition  of $W$ implies that
	\begin{equation}
	\label{eqD1.7.2}
	|W|=\left( b+\tfrac1{2}\right) ^{-1}.
	\end{equation}
	
	Let $P$ denote  the convex hull of the points 
	\begin{equation}
	\label{eqD1.7.4}
	(\tfrac1{2},-\tfrac1{2}),(\tfrac1{2},\tfrac1{2}),(a,b),(-\tfrac1{2},\tfrac1{2}),(-\tfrac1{2},-\tfrac1{2}),(-a,-b)
	\end{equation} 
	(see  the shaded region in Figure \ref{fig:polygon}).
	Observe that in the case where the  polygon $A$ is a parallelogram then $P$ is the unit cube, $P=[-\tfrac1{2},\tfrac1{2}]^2$, and otherwise $P$ is a hexagon of measure $ b+\frac1{2} $. 
	In any case,  it follows from \eqref{eqD1.7.2}  that
	\begin{equation}
	\label{eqD1.7.8}
	|W|=|P|^{-1}.
	\end{equation}
	
	Furthermore, as $A$ is convex and all the points in \eqref{eqD1.7.4} are vertices of $A$, the polygon $P$ is contained in $A$. We conclude that if $A$ is  a parallelogram or a hexagon, then it coincides with $P$ and so $|P|= |A|$; whereas otherwise, $P$ is strictly contained in $A$, and $|P|<|A|$. Combining the latter conclusion with \eqref{eqD1.7.8} completes the proof of our claim, as it shows that if $A$ is a parallelogram or a hexagon then 
	$|W|= |A|^{-1}$, and otherwise $|W|>|A|^{-1}$, as required.
\end{proof}

\subsection{}
It should be noted that the condition $|W| = |A|^{-1}$
in part \ref{thmD1.7.1} of \thmref{thmD1.7}
is sharp in the sense that the window $W$
cannot be chosen to have   measure strictly greater than $|A|^{-1}$. 
This is a consequence of the following lemma:

\begin{lem}
	\label{lemD1.6}
	Let $A$ be  a convex, centrally symmetric polygon in $\R^2$. If $A$ is a
	spectral set, then any window $W$ of $A$  satisfies $|W| \leq |A|^{-1}$.
\end{lem}

This result is analogous to \cite[Lemma 2.3]{LRW00}, where it was shown that
if $A$ is a spectral set and if $W$ is an orthogonal packing region for $A$,
then $|W| \leq |A|^{-1}$.

\begin{proof}[Proof of \lemref{lemD1.6}]
	Suppose that $W$ is a window for $A$, which means that \eqref{defD1.2} holds.
	Since $A$ is spectral,  \lemref{lemD1.5} allows us to choose  a spectrum $\Gam$ for $A$ which satisfies \eqref{eqD1.5}. 
	Thus, from   \eqref{eqD1.5} and \eqref{defD1.2}, it follows that
	$(\Gam-\Gam)\setminus\{0\}\subset\Delta(W)^\complement.$
	Observe that according to \lemref{lemC2.3}, this implies  that $W+\Gam$ is a packing, or equivalently, $\1_W+\Gam$ is a packing.
	On the other hand, part \ref{lemC1.7.1} of \lemref{lemC1.7} implies that if  $f:=\sqft{A}$, then $f+\Gam$ is a tiling.
	Now, as $f+\Gam$ is a tiling and $\1_W+\Gam$ is a packing, we may apply  part \ref{lemC1.6.1} of \lemref{lemC1.6} and deduce that
	\[|W|=\int \1_W\leq \int f=|A|^{-1},\]
	which completes the proof.
\end{proof}

One can think of  \lemref{lemD1.6} as imposing a necessary condition  
for the spectrality of a convex, centrally symmetric polygon $A \subset \R^2$. Namely, $A$ cannot
be spectral if it has a window $W$ such that $|W| > |A|^{-1}$. 
Using this proposition together with part \ref{thmD1.7.2} of \thmref{thmD1.7}
yields a new proof of the result from \cite{IKT03} 
which characterizes the spectral convex polygons in two dimensions:

\begin{corollary}[\cite{IKT03}]
\label{corD1.8}
Let  $A$ be a convex polygon  in $\R^2$. Then
 $A$ is a spectral set if and only if $A$ is either a parallelogram or a  centrally symmetric hexagon. 
\end{corollary}

\begin{proof}
We know that  parallelograms and centrally symmetric hexagons are spectral
sets (as these are  convex polygons that   tile by translations). Conversely, suppose that $A$ is a
spectral convex polygon. According to \thmref{thmC6.1}, $A$ must be centrally symmetric.
If $A$ is neither a parallelogram nor a  hexagon, then by
part \ref{thmD1.7.2} of \thmref{thmD1.7} there is 
a window $W$ for $A$, $|W| > |A|^{-1}$.
However this contradicts \lemref{lemD1.6}.
\end{proof}

\subsection{}
We can now complete the proof of our main result.

\begin{proof}[Proof of \thmref{thmA1.1}] 
We assume that $\Omega=A\times B$ is the product of a
convex polygon $A \subset \R^2$, and a bounded, measurable set  $B \subset \R^m$. 
We already know that the spectrality of both 
$A$ and $B$ implies the spectrality of $\Omega$. It remains therefore to prove the converse assertion,
 namely, if $\Omega$ is spectral then both
$A$ and $B$ must be spectral sets.

So suppose that $\Omega$ is spectral. Then according to \thmref{thmC5.4}, the convex polygon
$A$ must be centrally symmetric. By
\thmref{thmD1.7} we can find a window $W$ for $A$, such that  $|W| = |A|^{-1}$
if $A$ is either a parallelogram or a hexagon, and $|W| > |A|^{-1}$ otherwise.

Using \lemref{lemD2.1} we can find  a spectrum $\Lam$ for $\Omega$
which satisfies condition \eqref{eqD2.1}. As $W$ is a window for $A$,
it satisfies \eqref{defD1.2}. Combining these two conditions implies that
\[
	(\Lam-\Lam) \setminus \{0\} \subset (\Delta(W)^\complement \times \R^m) \cup (\R^2 \times \zeros(\ft{\1}_B)).
\]
That is, the spectrum $\Lam$ is $W$-compatible in the sense of \defref{defC4.1}.

We may therefore invoke \thmref{thmC4.3} in our present situation. It follows
from part \ref{thmC4.3.2} of this theorem that the set $B$ must be spectral,
so the spectrality of  $B$ is established.
To conclude that  $A$ must also be spectral, we can use
 part \ref{thmC4.3.1} of  \thmref{thmC4.3}, which yields that
$|W| = |A|^{-1}$. This is not possible unless $A$ is either
a parallelogram or a hexagon. In particular this implies
the spectrality of $A$, as we had to show.
\end{proof}

\subsection{}
Based on  \thmref{thmA1.1}  and the results obtained in \cite{GL16, GL17}
we can now deduce \corref{corA1.3}, which states that
  spectrality and tiling are equivalent properties
for decomposable convex polytopes in four dimensions.

\begin{proof}[Proof of \corref{corA1.3}] 
We assume that $\Omega$ is a convex polytope in $\R^4$, and that $\Omega$ is decomposable.
The decomposability assumption means that $\Omega$ 
can be mapped by an invertible affine transformation to a cartesian product $A \times B$ of two
convex polytopes $A \subset \R^n$,  $B \subset \R^m$  $(n,m \geq 1)$ where $n+m=4$.
By the invariance under affine transformations, it would
be  enough to consider the case when $\Omega = A \times B$. 

 We need to prove that $\Omega$ is spectral if and only if it can tile by translations.
It is already known that the convex polytopes which tile by translations
are spectral, and what has to be proved is that $\Om$ can be spectral
only if it tiles.

We therefore assume that $\Om$ is spectral.
We may suppose that $n \leq m$, which leaves two possibilities, $n=1$ and $m=3$, or $n=m=2$.

If  $n=1$ and $m=3$, then $A$ is an interval in $\R$, and $B$ is a convex polytope in $\R^3$
(so in this case, $\Omega$ is a prism with base $B$).  Using  \thmref{thmA1.0} we obtain
that $B$ is a spectral set.  Hence $B$ is a spectral
three-dimensional convex polytope, and so we know from \cite[Theorem~1.2]{GL17} that 
$B$ can tile $\R^3$ by translations.
Since the interval $A$ can obviously tile $\R$, it follows
that the product $\Omega = A \times B$  tiles $\R^4$ by translations, as we had to show.

Next we consider the remaining case, when  $n=m=2$.
In this case, both  $A$ and $B$ are convex polygons in $\R^2$,  hence we may apply
\thmref{thmA1.1}. It follows from the proof of this theorem 
that $A$ must be either a parallelogram or a  centrally symmetric hexagon,
and since the roles played by $A$ and $B$ are symmetric, the same
is true also for $B$. Hence each one of the sets $A, B$ can tile $\R^2$
by translations, which again implies that their product $\Omega$ tiles $\R^4$ by translations.
This concludes the proof.
\end{proof}


\section{Remarks} \label{secF1}

\subsection{}
It is a natural problem to extend Theorems \ref{thmA1.0}   and \ref{thmA1.1}  to higher dimensions.
\begin{problem}
	\label{probJ1.1}
 Let $A$ be a convex  polytope in $\R^n$, and $B$ be a bounded, measurable set in $\R^m$.
 Prove that their product $\Omega = A \times B$ is spectral if and only if $A$ and $B$ are both spectral sets.
\end{problem}

Theorems \ref{thmA1.0}   and \ref{thmA1.1} say that this is true for dimensions $n=1,2$.
For any dimension $n$, we know from \thmref{thmC5.4} that the spectrality of the product $\Omega = A \times B$ implies
that $A$  must be centrally symmetric and have centrally symmetric facets.

One may attempt to solve  \probref{probJ1.1} for dimensions $n \geq 3$
by adapting the approach   used in this paper for $n=2$. Such a
solution should involve two main steps:

(i) The set $H(A)$ should be defined in an appropriate way,
such that the corresponding versions of Lemmas  \ref{lemD1.5} 
and \ref{lemD2.1} would be true.

(ii)  A result analogous to \thmref{thmD1.7} should be proved, stating
that $A$ has  a window $W$ of measure $|W| \geq |A|^{-1}$, and 
moreover if $A$ is not spectral then
$W$ can be chosen such that $|W| > |A|^{-1}$. Here a ``window''  is again
defined by \eqref{defD1.2} but   with respect to the definition of $H(A)$ made in the previous step.

Once these two steps are accomplished,  a  proof of the assertion
in  \probref{probJ1.1}  can be completed using \thmref{thmC4.3}, in the same
way as we have done above for $n=2$.

\subsection{}
In dimension $n=3$, we are able to perform the first step in the above scheme.
That is, we can define the set $H(A)$ in a natural way, and then
prove the corresponding versions of Lemmas  \ref{lemD1.5} and \ref{lemD2.1}.
In what follows, we explain the definition of the set $H(A)$ in the three-dimensional setting.

Let $A \subset \R^3$ be a convex polytope, which is
 centrally symmetric and has centrally symmetric facets.  
We will assume that $A$ is not a prism. (If $A$ is a prism,
 then $A$ is decomposable, so in this case
\probref{probJ1.1} can be solved using  Theorems \ref{thmA1.0}   and \ref{thmA1.1}.)

Let $F$ be one of the facets of $A$. Then by the central symmetry of both $A$ and $F$,
the opposite facet $F'$ is a translate of $F$, hence there is a
translation vector $\tau_F$ which carries $F'$ onto $F$.
Further, if $e$ is an edge of $A$ which is contained in $F$, then  the central symmetry 
of $F$ implies that there is another edge $e'$ of $F$, which is parallel to $e$ and
has the same length. Let $\tau_{F,e}$ be the translation vector which carries $e'$ onto $e$.
Denote
\begin{equation}
	\label{eqdefJ3.1}
	H(A, F, e) := \{t \in \R^3 : \dotprod{t}{\tau_F} \in \Z
	\; \; \text{or} \; \;  \dotprod{t}{\tau_{F,e}} \in \Z
	\; \; \text{or} \; \;  \dotprod{t}{e} \in \Z \setminus \{0\} \}.
\end{equation}
Finally, we define
\begin{equation}
	\label{eqdefJ3.2}
H(A) := \bigcap_{(F,e)}  H(A, F, e) \setminus \{0\},
\end{equation}
where the intersection is taken over all the pairs $(F,e)$ such that
$F$ is a facet of $A$, and $e$ is an edge of $A$ which is contained in $F$.

(It turns out that the same set was used  also in the paper \cite{GKRS13}, where the structure of multi-tilings of $\R^3$
by translates of a convex polytope was studied.)

In \cite[Sections 6, 7, 12]{GL17} the following claim was
proved:  if $\Gamma$ is a spectrum for $A$, and
if $(F,e)$ is a pair as above, then there exists a sequence 
of translates of $\Gamma$ which convergences weakly 
to a spectrum $\Gamma'$ of $A$ that satisfies the condition
$\Gamma' - \Gamma' \subset H(A,F,e)$.
By iterating this process over all the pairs $(F,e)$ we can obtain:

\begin{lem}
	\label{lemJ4.1}
	Let $A \subset \R^3$ be a convex polytope, centrally symmetric
	and with centrally symmetric facets.  If
	$A$ is a spectral set, then it has a spectrum $\Gamma$ satisfying
\begin{equation}
	\label{eqJ4.2}
	(\Gamma - \Gamma) \setminus \{0\} \subset H(A).
\end{equation}
\end{lem}

This is the analog in dimension $n=3$ of \lemref{lemD1.5}.
In a similar way, we can also prove the corresponding version
of \lemref{lemD2.1}, namely:

\begin{lem}
	\label{lemJ4.3}
	Let $A$ be a convex polytope in $\R^3$, centrally symmetric
	and with centrally symmetric facets, and let $B$ be a
	bounded, measurable set in $\R^m$. If 
	$\Omega = A \times B$ is a  spectral set, 
	then there is a spectrum $\Lam$ of $\Omega$ such that
\begin{equation}
	\label{eqJ4.4}
		(\Lam-\Lam) \setminus \{0\}  \subset (H(A) \times \R^m) \cup (\R^3 \times \zeros(\ft{\1}_B)).
\end{equation}
\end{lem}

Actually, If some of the facets of $A$ happen to be quadrilateral, then the conclusions in
the last two lemmas can be somewhat improved, in the following sense:
 for each pair  $(F,e)$  such that $F$  is a quadrilateral facet of $A$,
 we can redefine $H(A,F,e)$ to be a set smaller than the one in \eqref{eqdefJ3.1},
but such that Lemmas \ref{lemJ4.1} and \ref{lemJ4.3} will  remain true. The new
definition of the set $H(A,F,e)$ in the case when $F$  is quadrilateral 
is obtained from  \eqref{eqdefJ3.1} by replacing the condition
$\dotprod{t}{\tau_{F,e}} \in \Z$ with the stronger one
$\dotprod{t}{\tau_{F,e}} \in \Z \setminus \{0\}$.

We conjecture that if $A \subset \R^3$ is a convex polytope, 
 centrally symmetric and with centrally symmetric facets, but $A$ is not a prism,
and if the set $H(A)$  is defined as above, then a corresponding version
of \thmref{thmD1.7} should be true. That is, if $A$ can tile the space
 by translations then it has a window $W$,  $|W| = |A|^{-1}$;
and otherwise, $A$ admits a window $W$ such that $|W| > |A|^{-1}$.
Such a result would imply a solution to \probref{probJ1.1} for
dimension $n=3$.  Moreover, it would provide an alternative approach to the
main result in \cite{GL17} which states that if a convex polytope in $\R^3$ is
spectral, then it can tile by translations. This will be the subject of a future work.


\end{document}